\documentclass[10pt,reqno]{amsart}
\usepackage{amsmath}
\usepackage{amsfonts}
\usepackage{indentfirst}
\usepackage{latexsym}
\usepackage{color}
\allowdisplaybreaks
\usepackage{amsthm}
\usepackage{amssymb}
\usepackage{enumerate}
\usepackage{ulem}
\allowdisplaybreaks

\newtheorem{theorem}{Theorem}[section]
\newtheorem{lemma}[theorem]{Lemma}
\newtheorem{thm}[theorem]{Theorem}
\newtheorem{prop}[theorem]{Proposition}
\newtheorem{cor}[theorem]{Corollary}
\newtheorem{defn}[theorem]{Definition}
\newtheorem{rmk}[theorem]{Remark}
\makeatletter

\newcommand{\Rmnum}[1]{\expandafter\@slowromancap\romannumeral #1@}
\makeatother
\begin{document}
\title[A Gromov-Hausdorff convergence theorem of surfaces in $\mathbb{R}^n$ with small total curvature]
{\bf A Gromov-Hausdorff convergence theorem of surfaces in $\mathbb{R}^n$ with small total curvature}
\author[ J. Sun, J. Zhou]
{Jianxin Sun, Jie Zhou}
\address{
\newline
Jianxin Sun:
Academy of Mathematics and Systems Science, CAS,
Beijing 100190, P.R. China.
{\tt Email:sunjianxin201008@126.com}
\newline
\newline
 Jie Zhou:
Academy of Mathematics and Systems Science, CAS,
Beijing 100190, P.R. China.
{\tt Email:zhoujie2014@mails.ucas.ac.cn}}
\date{}
\maketitle
\begin{abstract}
In this paper, we mainly study the compactness and local structure of immersing surfaces in $\mathbb{R}^n$ with local uniform bounded area and small total curvature $\int_{\Sigma\cap B_1(0)} |A|^2$. A key ingredient is a new quantity which we call isothermal radius. Using the estimate of the isothermal radius we establish a compactness theorem of such surfaces in intrinsic $L^p$-topology and extrinsic $W^{2,2}$-weak topology. As applications, we can explain Leon Simon's decomposition theorem\cite{LS} in the viewpoint of convergence and prove a non-collapsing version of H\'{e}lein's convergence theorem\cite{H}\cite{KL12}.

\end{abstract}

\section{Introduction}

Let $F:\Sigma\to \mathbb{R}^n$ be an immersion of the surface $\Sigma$ in $\mathbb{R}^n$. The total curvature of $F$ is defined by
$$\int_{\Sigma}|A|^2d\mu_g,$$
where $A$ is the second fundamental form and $g=dF\otimes dF$ is the induced metric. There are many results about the $L^p(p\ge 2)$ norm of the second fundamental form.

In\cite{L}, Langer first proved surfaces with $\|A\|_{L^p}\le C (p>2)$ are locally $C^{1}$-graphs over small balls(but with uniform radius) of the tangent spaces and deduced the compactness of such surfaces in the meaning of graphical convergence. Recently, Breuning considered a high dimensional generalization of Langer's theorem in \cite{B15}. He proved that an immersion $f:M^n\to \mathbb{R}^{n+l}$  with bounded volume and $\int_M|A|^pd\mu_g\le C(p>n)$ is a $C^{1,\alpha}$-graph in a uniform small ball for $\alpha<1-\frac{n}{p}$. This is a geometric analogue of the Sobolev embedding $W^{2,p}\to C^{1,1-\frac{n}{p}}$ if we regard the second fundamental form as the ``second derivative" of an immersing submanifold. In the critical case $p=n=2$, Leon Simon proved a decomposition theorem \cite[lemma 2.1]{LS} which says a surface with bounded volume and sufficient small total curvature is an almost flat Lipschitz graph outside of some small topological disks. Using this and noticing the total curvature is equal to the Willmore functional $\int_{\Sigma}|H|^2d\mu_g$ up to a constant for a closed surface with fixed topology, he got the existence of surfaces minimizing the Willmore functional.

Another important observation of immersing surfaces with finite total curvature is the compensated compactness phenomenon obtained in \cite{MS}. In general, the total curvature only controls the $L^1$-norm of the Gauss curvature $\int_{\Sigma}|K|d\mu_g$. But under the condition
 $$\int|A|^2\le 4\pi\varepsilon,$$
  M\"{u}ller and  \v{S}ver\'{a}k estimated the Hardy norm of $*Kd\mu_g$  and solved the equation
  $$-\triangle v=*Kd\mu_g$$
  such that $v\in L^{\infty}$.
  Using this $L^{\infty}$-estimation of the metric, E.Kuwert, R. Sch\"{a}tzle, Y.X. Li \cite{KS12}\cite{KL12} and T.Rivière \cite{Ri13} proved the compactness theorem of immersing maps $f:\Sigma_g\to \mathbb{R}^n$ with Willmore functional value
      $$Will(f)< 8\pi.$$
  With the compactness theorem, they gave an alternate approach of existence of the Willmore minimizer, see \cite{KS13}\cite{Ri14}. Their mainly observation is that the Willmore functional value will jump over the gap $8\pi$ as the complex structure diverges to the boundary of the moduli space $\mathcal{M}_g$. And once the complex structure $\phi_k:\Sigma_k\to \Sigma_g$ converges, they could regard $f_k\circ \phi_k$ to be conformal and consider the convergence of these mappings in weak $W^{2,2}_{loc}(\Sigma\backslash S,\mathbb{R}^n)$ topology.

Geometrically, the divergency of the complex structure means the collapsing of geodesics, which changes the local topology(or area density) and contributes a gap to the Willmore functional value. To use a geometric quantity to replace the convergence of the complex structure to rule out the collapsing phenomenon, we define the isothermal radius:
\begin{defn}[Isothermal radius]\label{Isothermal radius}
  Assume $(\Sigma, g)$ is a Riemannian surface with metric $g$. For $\forall p\in \Sigma$, define the isothermal radius of the metric $g$ at the point $p$ to be
   $$i_g(p)=\mathop{\mathrm{sup}}\{r|\exists\ \text{isothermal  coordinate} \  f_0:D_1(0)\to U(p)\ s.t.\ U(p)\supset B_r^{\Sigma}(p)\ \text{and} \ f_0(0)=p \},$$
   where $B_r^{\Sigma}(p)$ is the geodesic ball centered at $p$ with radius $r$. We call $i_g(\Sigma)=\inf_{p\in \Sigma}i_g(p)$ the isothermal radius of $(\Sigma,g)$.
  \end{defn}
  By the existence theorem of local conformal coordinates, $i(p)>0$ for any $p\in\Sigma$, but it may depend on the manifold $(\Sigma,g)$ and the point $p$. If we assume the isothermal radius  of a sequence of metrics has uniform lower bound, we can estimate the uniform bound of the metric and deduce the following compactness theorem:

\begin{thm}[Fundamental convergence theorem]\label{convergence}
  For any two positive real numbers $R, V$ and $\varepsilon<1$, and some positive function  $i_0:(0,R]\to \mathbb{R}_{>0}$, define $\mathcal{C}(n,\varepsilon,i_0, V, R)$ as the space of Riemannian immersions (of open Riemannian surfaces) $F:(\Sigma, g, p)\to \mathbb{R}^n$ proper in $B_R(0)$ such that $F(p)=0$, $F^{-1}(B_R(0))$ is connected,
   $$i_g(x)\ge i_0(R-|F(x)|)>0,\forall x\in\Sigma\cap F^{-1}(B_R(0)), $$
   $$\int_{\Sigma\cap F^{-1}(B_R(0))}|A|^2\le4\pi\varepsilon,$$
and
   $$\mu_g(\Sigma\cap F^{-1}(B_R(0)))\le VR^2,$$
where $B_R(0)$ is the open ball in $\mathbb{R}^n$ with radius $R$ and $i_g$ is the isothermal radius. Then for any sequence $\{F_k:(\Sigma_k, g_k,p_k)\to (\mathbb{R}^n,0)\}_{k=1}^{\infty}$ in $\mathcal{C}(n,\varepsilon, i_0,V,R)$, there exists a subsequence(also denoted as $\Sigma_k$) such that
  \begin{enumerate}[$1)$]
  \item $($Intrinsic convergence$)$ There is a pointed Riemannian surface $(\Sigma,g,p)$  $($may not complete$)$ with continuous Riemannian metric $g$ such that $(\Sigma_k\cap F_k^{-1}(B_R(p_k)), g_k, p_k)$ converges in pointed $L^p$ topology to $(\Sigma,g,p)$, that is, there exist smooth embeddings $\Phi_k:\Sigma\to \Sigma_k$, s.t. $\Phi_k^*g_k\to g$ in $L^p_{loc}(d\mu_g)$.
  \item The complex structure $\mathcal{O}_k$ of $\Sigma_k$ converges locally to the complex structure $\mathcal{O}$ of $\Sigma$.
  \item $($Extrinsic convergence$)$  There exists a $\mathbf{proper}$ isometric immersion $F:(\Sigma,g,p)\to B_R(0)$ such that $F(p)=0$ and $F_k\circ \Phi_k$ converges to $F$ weakly in $W^{2,2}_{loc}(\Sigma,\mathbb{R}^n)$ and strongly in $W^{1,p}_{loc}(\Sigma,\mathbb{R}^n)$.
  \end{enumerate}
\end{thm}
 Moreover, just as Anderson noticed in \cite[main Lemma 2.2]{A},  Theorem \ref{convergence} can feed back to provide a lower bound estimate of the isothermal radius.

   \begin{prop}\label{isothermal radius}
 For any fixed $V\in \mathbb{R}_{+}$, there exist $\varepsilon_0(V)>0$ and $\alpha_0(V)>0$ such that for any properly immersed Riemannian surface $F:(\Sigma,g,p)\to(\mathbb{R}^n,0)$ satisfying
 $$\mu_g(\Sigma\cap B_1(0))\le V\text{ and }\int_{\Sigma\cap B_1(0)}{|A|}^2d\mu_g\le \varepsilon_0(V),$$
  we have
 $$i_g(x)\ge \alpha_0(V), \forall x\in \Sigma\cap B_{\frac{1}{2}}(0),$$
  where $B_1(0)=\{y\in \mathbb{R}^n||y|<1\}$ and by $\Sigma\cap B_1(0)$ we actually mean the connected component of $\Sigma\cap F^{-1}(B_1(0))$ containing $p$.
   \end{prop}
    As an application, we use the fundamental convergence theorem to give a blowup approach for Leon Simon's decomposition theorem. A key point is to use a blowup argument to
 reduce the immersing case to the embedding case. Another key point is to use the Poincar\'e inequality to estimate the area of the surface out of the multi-graph, which finally guarantees the multi-graph must be single, see Theorem \ref{decomposition} for more details.

   As another application of the estimation of the isothermal radius, we can prove a non-collapsing version of H\'{e}lein's  compactness theorem. In \cite{H}, H\'{e}lein proved that a sequence of $W^{2,2}$-conformal immersions with uniform bounded Willmore energy will converge weakly to a limit $f_{\infty}$ in $W^{2,2}_{loc}(\Sigma\backslash S)$ outside of finite many singularities $S\subset \Sigma$. But the conformal invariant property of the Willmore energy and the non-compactness of both the intrinsic and extrinsic conformal groups $\mathcal{M}(\Sigma)$  and $\mathcal{M}(\mathbb{R}^n)$ may cause the limit $f_{\infty}$ to collapse. In the case $\Sigma=S^2$, if we assume additionally $vol_{g_k}(S^2)\equiv 1$, then we can use the estimation of isothermal radius to exclude the collapsing caused by the extrinsic conformal group $\mathcal{M}(\mathbb{R}^n)$, see Corollary \ref{non-collapsing}. This result is also obtained in \cite{CL14} by the bubble tree convergence argument.

   The article is organized as following. In sect.\ref{section:prelimaries}, we prepared some important preliminaries. In sect.\ref{section:metric estimate}, we use the lower bound of the isothermal radius to estimate the metric. In sect.\ref{section:convergence}, we prove the fundamental convergence theorem. In sect.\ref{section:estimate the isothermal radius} and sect.\ref{section:non-collapsing Helein}, we prove Proposition \ref{isothermal radius} and give the two applications.

\section{Preliminaries and Notations}\label{section:prelimaries}
\subsection{Hardy estimate}
   In this section, we list some theorems we will quote in this paper. The first is M\"{u}ller and  \v{S}ver\'{a}k's Hardy-estimate \cite[Corollary 3.5.7.]{MS} mentioned above. Under a local isothermal coordinate, we assume $g=e^{2u}(dx^2+dy^2)$, and the Gauss curvature equation is
 $$-\triangle u=*Kd\mu_g=*G^*\omega,$$
 where $G$ is the Gauss map $G:\Sigma\to \mathbb{CP}^{n-1}$, $G(p)=[\frac{e_1(p)+\sqrt{-1}e_2(p)}{2}]$, $\{e_1,e_2\}$ are orthonormal basis of $\Sigma$ at the point $p$ and $\omega$ is  the K\"{a}ller form on $\mathbb{CP}^{n-1}$. In \cite{MS}, when observing that $\omega$ has the algebraic structure of determinant when transgressed to the total space $S^{2n-1}$ of the Hopf fibration $\pi: S^{2n-1}\to \mathbb{CP}^{n-1}$,  M\"{u}ller and  \v{S}ver\'{a}k improved the regularity of $*Kd\mu_g$ from $L^1$ to $\mathcal{H}^1$, the Hardy space, by using the results of  Coifman, Meyer, Lions and Semmes\cite{CLMS}(see also \cite{M}) under the condition
 $$\int|A|^2\le 4\pi\varepsilon$$
 for some $\varepsilon\in (0,1).$
  Moreover, when noticing that in dimension two, the fundamental solution $ln|x|$ belongs to $BMO$, the dual of Hardy space (\cite{FS}), they can solve $-\triangle v=Ke^{2u}$ with $L^{\infty}$-estimation:
   \begin{thm}\label{MSFS}
      For $0<\varepsilon<1$. Assume $\varphi\in W_0^{1,2}(\mathbb{C},\mathbb{CP}^{n})$ satisfies  $\int_{\mathbb{C}}\varphi^*\omega=0$ and that $\int_{\mathbb{C}}|D\varphi\wedge D\varphi|\le2\pi\varepsilon$. Then $*\varphi^*\omega\in\mathcal{H}^1(\mathbb{C})$ with $\|\varphi^*\omega\|_{\mathcal{H}^1}\le c_1C(n,\varepsilon)\|D\varphi\|_{L^2}^2$, where $C(n,\varepsilon)=1+\frac{4n^2(1-{\varepsilon}^{\frac{1}{n}})}{(1-\varepsilon)^2}$. Moreover, the equation $-\triangle u=*\varphi^*\omega$ admits a unique solution $v:\mathbb{C}\to\mathbb{R}$ which is continuous and satisfies:
     $$\mathop{\mathrm{lim}}_{z\to \infty}v(z)=0,$$
     and
     $$\int_{\mathbb{C}}|D^2v|+\{{\int_{\mathbb{C}}{|Dv|}^2}\}^{\frac{1}{2}}+\mathop{\mathrm{max}}_{z\in \mathbb{C}}|v|(z)\le c_2\|\varphi^*\omega\|_{\mathcal{H}^1} \le c_3 C(n,\varepsilon)\|D\varphi\|_{L^2}^2,$$
     where $c_1, c_2$ and $c_3$ are constants independent of $n$ and $\varepsilon$, which will be denoted as a same notation $c$ in the following text.
   \end{thm}
   \subsection{Monotonicity formulae}
   We will also quote Leon Simon's monotonicity formulae \cite{LS}(see also \cite{KR}) in the blowup argument of estimating the isothermal radius. It is also used in our application \Rmnum{2}.
   \begin{thm}\label{monotonicity formulae}
   Assume $\mu\neq 0$ is an integral 2-varifold in an open set $U\subset \mathbb{R}^n$ with square integrable weak mean curvature $H_{\mu}\in L^2(\mu)$ and $B_{\rho_0}(x_0)\subset\subset U$ for some $x_0\in\mathbb{R}^n$ and $\rho_0>0$. Then for $\forall \delta>0$ and $0<\sigma\le\rho\le\rho_0$, we have
     $$\sigma^{-2}\mu(B_{\sigma}(x_0))\le(1+\delta)\rho^{-2}\mu(B_{\rho}(x_0))+(1+\frac{1}{4\delta})W(\mu),$$
   where $W(\mu)=\frac{1}{4}\int_U{|H_{\mu}|}^2d\mu$ is the Willmore energy of $\mu$.
   \end{thm}
   Here we explain some notations we use in this paper. For a Riemannian immersion $F:(\Sigma,g)\to \mathbb{R}^n$, $p\in \Sigma$, $x\in \mathbb{R}^n$ and $R>0$, we denote
   $$B_R(x)=\{y\in \mathbb{R}^n||y-x|<R\}, \ \ \ \ B^{\Sigma}_R(p)=\{q\in \Sigma| d_g(q,p)<R\},$$
   and
   $$\Sigma^{R}(p):=\text{ the connected component of } F^{-1}(B_R(F(p))) \text{ containing } p.$$
   We will also abuse the inaccurate but intuitionistic notation $\Sigma\cap B_R(F(p))$ to denote $\Sigma^{R}(p)$.
   \section{$L^{\infty}$--estimate of the metric under isothermal coordinates}\label{section:metric estimate}
\subsection{}
   Basic setting: for an immersing Riemannian surface $F:\Sigma\to \mathbb{R}^n$ with the induced metric $g=dF\otimes dF$, by the existence theorem of isothermal coordinate \cite{Chern 1} and Riemann mapping theorem, there exists an isothermal coordinate $f_0:D_1(0) \to U\subseteq \Sigma \to \mathbb{R}^n$  for a small neighborhood of p such that $f_0(0)=p$, where $D_1=\{z=x+\sqrt{-1}y \in \mathbb{C}:|z|<1\}$ is the unit disk on the complex plain and the scale of $U$ may depend on $(\Sigma,g,p)$. This means $f_0:D_1\to \mathbb{R}^n$ is a conformal immersion, i.e., $f_0^*g=e^{2u(z)}(dx^2+dy^2)$. Furthermore, we have $f_0^*d\mu_g=e^{2u}dx\wedge dy$, $G^*\omega=Kd\mu_g=Ke^{2u}dx\wedge dy$ and ${|DG|}^2_g=\frac{1}{2}{|A|}^2_g$, where $G:\Sigma\to \mathbb{CP}^{n-1}$ is the Gauss map, $\omega$ is the K\"ahler form on $\mathbb{CP}^{n-1}$ and $A$ is the second fundamental form of the immersion $F:\Sigma \to \mathbb{R}^n$. As a consequence, if we let $\varphi=G\circ f_0$, then we get the basic Gauss curvature equation:
   $$-\triangle u=Ke^{2u}=*f_0^*(Kd\mu_g)=*f_0^*G^*\omega=*\varphi^*\omega$$.
   \begin{lemma}\label{extend}
   If $\int_{f_0(D_1)}{|A|}^2_g<+\infty$ and $\int_{D_1}|J_{\varphi}|dx\wedge dy=\int_{D_1}|D\varphi\wedge D\varphi|\le\pi\varepsilon$ for some $\varepsilon\in(0,1)$, then $Ke^{2u}=*\varphi^*\omega$ could be extended to be a function $w\in\mathcal{H}^1(\mathbb{C})$ with
     $$ \|w\|_{\mathcal{H}^1(\mathbb{C})}\le c(n,\varepsilon)\int_{f_0(D_1)}{|A|}^2_g.$$
   Furthermore, the equation $-\triangle u=w$ admits a unique solution $v=-\triangle^{-1}w$ which is continuous on $\mathbb{C}$ and satisfies:
     $$\mathop{\mathrm{lim}}_{z\to \infty}v(z)=0,$$
   and
     $$\int_{\mathbb{C}}|D^2v|+\{{\int_{\mathbb{C}}{|Dv|}^2}\}^{\frac{1}{2}}+\mathop{\mathrm{max}}_{z\in \mathbb{C}}|v|(z)\le c\|w\|_{\mathcal{H}^1} \le c(n,\varepsilon)\int_{f_0(D_1)}{|A|}^2_g.$$
  \end{lemma}
  \begin{proof}
    Set $\varphi=G\circ f_0:D_1\to\Sigma\to\mathbb{CP}^{n-1}$, and define $\bar{\varphi}:\mathbb{C}\to\mathbb{CP}^{n-1}$ by
    $$
    \bar{\varphi}(z)=
    \begin{cases}
    \varphi(z), z\in D_1,\\
    \varphi(\frac{1}{\bar{z}}), z\in D_1^*=\{z||z|\ge1\}.
    \end{cases}
    $$
    Then we have,
    $$\int_{\mathbb{C}}\bar{\varphi}^*\omega=\int_{D_1}\varphi^*\omega-\int_{D_1}\varphi^*\omega=0,$$
    $$\int_{\mathbb{C}}|D\bar{\varphi}\wedge D\bar{\varphi}|dxdy=2\int_{D_1}|D\bar{\varphi}\wedge D\bar{\varphi}|dxdy\le 2\pi\varepsilon,$$
    and
    \begin{align*}
    \int_{\mathbb{C}}{|D\bar{\varphi}|}^2dxdy=2\int_{D_1}{|D\varphi|}^2dxdy\le\int_{f_0(D_1)}{|A|}^2_gd\mu_g.
    \end{align*}
    The work is done if we define $w:=*\bar{\varphi}^*\omega$ and apply Theorem~\ref{MSFS}.
  \end{proof}
  \begin{rmk}
    If $\int_{f_0(D_1)}{|A|}^2d\mu_g\le4\pi\varepsilon$, then
    \begin{align*}
    \int_{D_1}|D\varphi\wedge D\varphi|dxdy=\int_{D_1}\sqrt{\mathrm{det}{(D\varphi)}^*D\varphi}dxdy
    \le\frac{1}{4}\int_{f_0(D_1)}{|A|}^2d\mu_g\le\pi\varepsilon.
    \end{align*}
     So, in the following text, we may use $\int_{f_0(D_1)}{|A|}^2d\mu_g\le4\pi\varepsilon$ to replace the conditions $\int_{f_0(D_1)}{|A|}^2_g<+\infty$ and $\int_{D_1}|J_{\varphi}|dx\wedge dy=\int_{D_1}|D\varphi\wedge D\varphi|\le\pi\varepsilon$.
  \end{rmk}
  The following corollary is simple but important in the estimation of the isothermal radius.
   \begin{cor}\label{uniformly convergence}
    In the case $\int_{\Sigma_k}{|A_k|}^2_{g_k}d\mu_{g_k}\to 0$, $v_k$ converges to $0$ uniformly on the unit disk $D_1$.
  \end{cor}
  The following lemma shows that the lower bound of isothermal radius could rule out the collapsing phenomenon and dominate the metric uniformly.
  \begin{lemma}\label{lemma3}
    Take a local isothermal coordinate $f_0:(D_1,0)\to(U\subset\Sigma,p)$ and denote $f_0^*g=e^{2u(z)}(dx^2+dy^2)$ as before. Assume $\int_{f_0(D_1)}{|A|}^2d\mu_g\le A_0$, $\int_{D_1}|D\varphi\wedge D\varphi|dxdy\le\pi\varepsilon$, $\mathop{\mathrm{area}}_g(f_0(D_1))\le V$ and $d_g(p,f_0(D_1))\ge i_0>0$. Then $u\in L_{loc}^{\infty}(D_1)$ and for any $r\in(0,1)$, there exists a $C=C(n,\varepsilon, i_0, V, A_0,r)$ such that
    $$\mathop{\mathrm{sup}}_{D_r}|u|\le C(n,\varepsilon, i_0,V,A_0,r).$$
  \end{lemma}
  \begin{proof}
    Let $v=-\triangle w$ and $h=u-v$ as above. Then $h$ is harmonic and it is equal to estimate $u$ and $h$ since $v$ has been estimated by Lemma~\ref{extend}.\\
    Step 1. Estimate the upper bound. $\forall x\in D_1$, let $r(x)=d(x,\partial D_1)=1-|x|>0$. Then Jensen's inequality and mean value theorem for harmonic functions imply
    \begin{align*}
    u(x)&=\bar{u}(x)-\bar{v}(x)+v(x)\\
        &\le\frac{1}{\pi r^2(x)}\int_{D_{r(x)}(x)}u(y)+2C(n,\varepsilon)A_0 \\
        &\le\frac{1}{\pi r^2(x)}\ln{\int_{D_{r(x)}(x)}e^{2u(y)}}+2C(n,\varepsilon)A_0\\
        &\le\frac{\ln{V}}{\pi {(1-r)}^2}+2C(n,\varepsilon)A_0=:C(n,\varepsilon,V,A_0,r)
    \end{align*}
    for $x\in D_r$, where $\bar{h}(x)=\frac{1}{\pi r^2(x)}\int_{D_{r(x)}(x)}h(y)$ and $\bar{u}$, $\bar{v}$ are defined similarly.\\
     Step 2. The argument developed in \cite{Jingyi Chern Yuxiang Li} gives a control of $\bar{u}(0)$ from below by $d_g(p,f_0(D_1))$, i.e.,  $\bar{u}(0)\ge -C$ for some constant $C=C(n,\varepsilon,i_0,V,A_0)>0$. We argue by contradiction.
      If there exist a sequence of Riemannian immersions $F_k:({\Sigma}_k, g_k, p_k)\to \mathbb{R}^n$  which admit isothermal coordinates $f_k:(D_1,0)\to (U_k,p_k)$ and satisfy $$\int_{{\Sigma}_k\cap f_k(D_1)}{|A_k|}^2d\mu_{g_k}\le A_0, \int_{D_1}|D{\varphi}_k\wedge D{\varphi}_k|dxdy\le\pi\varepsilon,$$
     $$\mathop{\mathrm{area}}_{g_k}({\Sigma}_k\cap f_k(D_1))\le V,\ d_{g_k}(p_k, \partial U_k)\ge i_0,$$
     but ${\bar{u}}_k(0)=\frac{1}{\pi}\int_{D_1}u_k=C_k\to-\infty$, where $u_k$ is defined by $f_k^*g_k=e^{2u_k(z)}(dx^2+dy^2)$. Then, Lemma~\ref{extend} and step 2. imply
     $$h_k(0)={\bar{h}}_k(0)={\bar{u}}_k(0)-{\bar{v}_k}(0)\le{\bar{u}}_k(0)+C(n,\varepsilon)A_0\to-\infty$$
     and
     $$h_k(x)=u_k(x)-v_k(x)\le C(n,\varepsilon, A_0, V, r)\  \mathrm{for} \ x\in D_r(0).$$
     Now, Harnack's inequality and Lemma~\ref{extend} again imply $u_k=h_k+v_k\rightrightarrows-\infty$ on $D_r$,  which means $\lim\limits_{k\to\infty}\mathrm{area}_{g_k}(D_r)=0$ , for any $r\in(0,1)$. But on the other hand, for $\delta<{\delta}_0(n,\varepsilon, A_0,V,i_0)$ small enough, $d_{g_k}(p_k, \partial U_k)\ge i_0$ implies that $\mathrm{for} \ \forall x\in \partial D_1(0)$,
     \begin{align*}
     i_0&\le d_{g_k}(x,0)\\
          &\le l_{g_k}(\gamma_{0,x}) \ \ \ (\gamma_{0,x}(t)=tx:[0,1]\to D_1)\\
          &=\int_0^1e^{u_k(tx)}|x|dt\\
          &=\int_0^{\delta}e^{u_k(tx)}dt+\int_{\delta}^{1-\delta}e^{u_k(tx)}dt+\int_{1-\delta}^1 e^{u_k(tx)}dt\\
          &\le\delta e^{\frac{\ln{V}}{2\pi {(1-\delta)}^2}+2C(n,\varepsilon)A_0}
          +{\Big(\int_{\delta}^{1-\delta}e^{2u_k(tx)}tdt\Big)}^{\frac{1}{2}}
          {\Big(\int_{\delta}^{1-\delta}\frac{1}{t}dt\Big)}^{\frac{1}{2}}\\
          &+{\Big(\int_{1-\delta}^1e^{2u_k(tx)}tdt\Big)}^{\frac{1}{2}}
          {\Big(\int_{1-\delta}^1\frac{1}{t}dt\Big)}^{\frac{1}{2}}\\
          &\le\frac{i_0}{2}
          +{\Big(\int_{\delta}^{1-\delta}e^{2u_k(tx)}tdt\Big)}^{\frac{1}{2}}
          {\Big(\ln{\frac{1-\delta}{\delta}}\Big)}^{\frac{1}{2}}\\
          &+{\Big(\int_{1-\delta}^1e^{2u_k(tx)}tdt\Big)}^{\frac{1}{2}}
          {\Big(-\ln{(1-\delta)}\Big)}^{\frac{1}{2}},
     \end{align*}
   which further implies
    \begin{align*}
    \big(\frac{i_0}{2}\big)^2\le
    2\Big(\int_{\delta}^{1-\delta}e^{2u_k(tx)}tdt\Big)\Big(\ln{\frac{1-\delta}{\delta}}\Big)
    +2\Big(\int_{1-\delta}^1e^{2u_k(tx)}tdt\Big)\Big(-\ln{(1-\delta)}\Big).
    \end{align*}
    Integrating this on $\theta=x\in[0,2\pi]$, we get
    \begin{align*}
    \frac{2\pi i_0^2}{4}
    &\le 2\int_0^{2\pi}\bigg(\Big(\int_{\delta}^{1-\delta}e^{2u_k(tx)}tdt\Big)\Big(\ln{(1-\delta)}-\ln{\delta}\Big)\\
    &+\Big(\int_{1-\delta}^1e^{2u_k(tx)}tdt\Big)\Big(-\ln{(1-\delta)}\Big)\bigg)d\theta\\
    &\le 2\mathrm{area}_{g_k}(D_{1-\delta}\backslash D_{\delta})(-\ln{\delta})+2V(-\ln{(1-\delta)}).
    \end{align*}
    Take $\delta_0$ small again and we get
    $$\mathrm{area}_{g_k}(D_{1-\delta}\backslash D_{\delta})\ge-\frac{\pi i_0^2}{8\ln{\delta}}>0$$
    for $\delta<{\delta}_0(n,\varepsilon, A_0,V,i_0)$. This contradicts to $\lim\limits_{k\to\infty}\mathrm{area}_{g_k}(D_r)=0$.\\
    Step 3. Finally, Lemma~\ref{extend}, Step 2 and Step 3 imply
    $$\mathop{\mathrm{max}}_{x\in{D_r}}h(x)\le C(n,\varepsilon, A_0,V,r) \  \mathrm{and} \  h(0)>-C(n,\varepsilon, A_0, V, i_0, r).$$
    Using Harnack's inequality and Lemma~\ref{extend} again, we know
    $$u(x)\ge-C(n,\varepsilon, A_0, V, i_0, r), \forall x \in D_r.$$
  \end{proof}
  The following corollary follows immediately since $h$ is harmonic.
  \begin{cor}\label{estimate for h}
  Under the above setting, $\|h\|_{C^k(D_r)}\le C(n,\varepsilon, A_0,V,i_0,r, k)$ for any integer $k$.
  \end{cor}
\section{Intrinsic and extrinsic convergence}\label{section:convergence}
 In this section, we first use Gromov's compactness theorem and the estimations established in the last section to prove an intrinsic convergence theorem and then use the $W^{2,2}$-estimate for the mean curvature equation to prove an extrinsic convergence theorem. They together form the fundamental convergence Theorem \ref{convergence}.
  \subsection{Intrinsic convergence}\label{Intrinsic convergence}

  \begin{defn}
  \begin{align*}
  \mathcal{E}(n,\varepsilon,i_0,A_0,V)=\{F:(\Sigma,g,p)\to\mathbb{R}^n\  \text{properly immersing} | i_g(\Sigma)\ge i_0>0,\\
   \int_{D_1}|D\varphi \wedge D\varphi|\le \pi\varepsilon,
   \int_{U(p)}{|A|}^2\le A_0,\  \mathrm{\mathrm{area}}_g(U(p))\le V,\forall p\in \Sigma\}
  \end{align*}
  where $\varphi=G\circ f_0$ is the Gauss map in the coordinate defined as in the beginning of last section and $f_0:D_1(0)\to U(p)$ is defined as in the definition of the isothermal radius.
  \end{defn}
  We conclude the lemmas in the last section as below.
  \begin{lemma} \label{L^infty}
  For any $(F:(\Sigma,g,p)\to \mathbb{R}^n)\in\mathcal{E}(n,\varepsilon,i_0,A_0,V)$ and $r\in(0,1)$ and $r\in(0,1)$,
  there exists a constant $C=C(n,\varepsilon,i_0,A_0,V,r)$ s.t.
  $$\|u\|_{L^{\infty}(D_r)}+\|Du\|_{L^2(D_r)}\le C(n,\varepsilon,i_0,A_0,V,r),$$
  where $u$ is defined by $f_0^*g=e^{2u}g_0$ and $f_0:D_1(0)\to U(p)$ is the isothermal coordinate defined above.
  \end{lemma}
  The following theorem is a modification of the Fundamental theorem in Petersen's book\cite{P}(See also \cite{C}) in a weaker regularity condition.
  \begin{thm}[Compactness theorem for immersed Riemannian surfaces]\label{intrinsic convergence} Any pointed sequence $\{F_k:(\Sigma_k,g_k,p_k)\to\mathbb{R}^n\}_{k=1}^{\infty}$  in $\mathcal{E}(n,\varepsilon, i_0,A_0,V)$ admits a subsequence$($still denoted as $\{(\Sigma_k,g_k,p_k)\}_{k=1}^{\infty})$ which converges to a complete Riemannian surface $(\Sigma,g,d,p)$ with a continuous Riemannian metric $g$ and a compatible metric $d\sim d_g$ in the following sense
  \begin{enumerate}[$(a)$]
  \item $(\Sigma_k,g_k,p_k)$ converges to $(\Sigma, d,p)$ in pointed Gromov-Hausdorff topology;
  \item $(\Sigma_k,p_k)$ converges to $(\Sigma,p)$ as pointed Riemannian surface$($i.e., the complex structure $($hence the smooth structure$)$ converges$)$;
  \item $\forall p\in(1,+\infty)$, $(\Sigma_k,g_k,p_k)$ converges to $(\Sigma, g,p)$ in pointed-$L^p$ $($hence a.e.$)$ topology.
  \item In general, we have $d\le d_g\le Cd$. But in the case $\int_{\Sigma_k}|A_k|^2d\mu_{g_k}\to 0$, we have $(\Sigma_k,g_k,p_k)$ converges to $(\Sigma,g,p)$ in $C^0_{loc}$ topology and $d_g=d$.
  \end{enumerate}
  \end{thm}
  Before proving this theorem, we make some conceptions clear.
  \begin{defn}[Convergence of complex(differential) structure] Assume $\{(\Sigma_k,p_k)\}_{k=1}^{\infty}$ and $(\Sigma,p)$ are all pointed Riemannian($C^{\infty}$, $C^{m+1,\beta}$) surfaces admitting complex($C^{\infty}$, $C^{m+1,\beta}$) structures
  \begin{align*}
  \mathcal{O}_k&=\{f_{ks}:D_r(0)\to U_{ks}\subset\Sigma_k\}_{s=1}^{\infty},\ \bigcup_{s=1}^{\infty}U_{ks}=\Sigma_k,\\
  \mathcal{O}&=\{f_{s}:D_r(0)\to U_{s}\subset\Sigma\}_{s=1}^{\infty},\ \bigcup_{s=1}^{\infty}U_{s}=\Sigma
  \end{align*}
  respectively. (WLOG, we assume $p_k\in U_{k1}$ and $p\in U_1$.) We call the complex ($C^{\infty}$, $C^{m+1,\beta}$) structure $\mathcal{O}_k$ converges to $\mathcal{O}$ in pointed $C^{\omega}(C^{\infty},C^{m+1,\beta})$ topology if  $\forall K\subset\subset \mathrm{Dom}(f_s^{-1}\circ f_t)$, $\exists k_0$ large enough, $\forall k\ge k_0$,
  $$K\subset\subset\mathrm{Dom}(f_{ks}^{-1}\circ f_{kt})\ \ \ \ \mathrm{and}\ \ \ \  f_{ks}^{-1}\circ f_{kt}\xrightarrow[]{C_c^{\omega}(C_c^{\infty},C_c^{m+1,\beta})}f_s^{-1}\circ f_t.$$
  In this case, we also call the sequence$\{(\Sigma_k,p_k)\}_{k=1}^{\infty}$ converges to $(\Sigma,p)$ as pointed Riemannian($C^{\infty}$, $C^{m+1,\beta}$) surface.
  \end{defn}
  \begin{rmk}
  Assume $\mathcal{O}_k$ and $\mathcal{O}$ are complex structures. Then by Montel's theorem, $\mathcal{O}_k\to\mathcal{O}$ in $C^{\omega}$ topology $\Leftrightarrow$ $\mathcal{O}_k\to\mathcal{O}$ in $C^0$ topology(after passing to a subsequence).
  \end{rmk}
  \begin{defn}[$L^p$ convergence]
  \begin{enumerate}[$1)$]
  \item Assume $A$ is a compact set in a differential manifold $M$, $\{f_k\}_{k=1}^{\infty}$ and $f$ are functions on $M$. We call $f_k$ converges to $f$ in $L^p(A)$ (a.e.) if there exists a finite coordinate cover $\{x_s:U_s\to\tilde{U}_s\subset\mathbb{R}^n\}_{s=1}^{N}$ with $\cup_{s=1}^N U_s\supset A$ s.t.
       $$\tilde{f}_{ks}=f_k\circ x_s^{-1} \xrightarrow[]{L^p_{loc}(\tilde{U}_S)}\tilde{f}_s=f\circ x_s^{-1}, \forall 1\le s\le N.$$
       It is easy to check this definition does not depend on the choice of coordinates.
  \item For Riemannian metrics $g_k$ and $g$ on $M$, we say $g_k\xrightarrow[]{L^p(A)}g$ if for the above coordinates $x_s$, the coefficients of the tensors defined by $x_s^{-1*}g_k=g^s_{k\alpha\beta}dx^{\alpha}dx^{\beta}$ and $x_s^{-1*}g=g^s_{\alpha\beta}dx^{\alpha}dx^{\beta}$ satisfies $$\sum_{1\le\alpha,\beta\le n}\|g^s_{k\alpha\beta}-g^s_{\alpha\beta}\|_{L^p(K)}\to 0,\forall 1\le s\le N, \forall K\subset\subset \tilde{U}_s.$$
  \item Assume $(M_k,g_k,p_k)$ and $(M,g,p)$ are complete Riemannian manifolds with continuous Riemannian metrics. We say $(M_k,g_k,p_k)\xrightarrow[]{pointed-L^p}(M,g,p)$ if $\forall R >0,\exists\Omega\supset B^{M}_R(p)\subset M$ and embedding $F_k:\Omega\to M_k$ for $k$ large enough s.t. $B^{M_k}_R(p_k)\subset F_k(\Omega)\subset M_k$ and $$F_k^*g_k\xrightarrow[]{L^p(\Omega)}g.$$
   \end{enumerate}
  \end{defn}
  \begin{proof}[Proof of Theorem~\ref{intrinsic convergence}]
 \  \\Step 1. (Gromov-Haustorff convergence)

   In the first paragraph we omit the footprint $k$. Let $f_0:D_1(0)\to U(p)$ be an isothermal coordinate and define $u$ by  $f_0^*g=e^{2u}g_0$. Then, we have $\|u\|_{L^{\infty}(D_r)}\le C(r)=C(n,\varepsilon, i_0, A_0,V,r)$ by Lemma~\ref{L^infty}, hence for any curve $\gamma:[0,1]\to D_r$ with $\gamma(0)=x,\gamma(1)=y$, we have
   $$e^{-C(r)}l_{g_0}(\gamma)\le l_g(f_0(\gamma))\le e^{C(r)}l_{g_0}(\gamma),$$
   which means
  \begin{enumerate}[(a)]\setlength{\itemsep}{-\itemsep}
  \item $d(f_0(x),f_0(y))\le e^{C(r)}|x-y|, \forall x, y \in D_r(0)$,\\
  \item $d(f_0(x),f_0(y))\ge e^{-C(r)}\mathop{\mathrm{min}}\{|x-y|,2r-|x|-|y|\}, \forall x,y \in D_r(0).$
  \end{enumerate}
  The same argument as in \cite[sec 10.3.4]{P} implies the capacity estimate, i.e., $\forall R>0$ and $\forall ((\Sigma,g,p)\to\mathbb{R}^n)\in \mathcal{E},\exists N(\alpha)=N(\alpha,R,r,C(r))$ and $\delta=\frac{1}{10}e^{-C(\frac{r+1}{2})}r$ s.t. $\mathop{\mathrm{Cap}_{B_R(p)}(\alpha)}\le N(\alpha)$ for $\forall \alpha\le{\alpha}_0=\delta$, where the capacity $\mathop{\mathrm{Cap}_{B_R(p)}(\alpha)}$ is defined by $$\mathop{\mathrm{Cap}_X(\alpha)}=\mathrm{ maximum\ number\ of\ disjoint}\ \frac{\alpha}{2}-\mathrm{balls\  in }\  X$$ for compact metric space $X$.
  So, for a sequence $\{(\Sigma_k,g_k,p_k)\to\mathbb{R}^n\}_{k=1}^{\infty}\subset\mathcal{E}$, choosing conformal coordinates covering $\{f_{ks}:D_1(0)\to U_{ks}(p^k_s)\subset \Sigma_k\}_{k,s=1}^{\infty}$ s.t. $f_{ks}(0)=p_s^k$, $p^k_1=p_k$, $U_{ks}\supset B^{\Sigma_k}_{i_0}(p^k_s)$, $\Sigma_k=\cup_s f_{ks}(D_r(0))$ and $B_{l\cdot\frac{\delta}{2}}(p_k)\subset\cup_s^{N^l} f_{ks}(D_r(0))$, Gromov's compactness theorem\cite{G}(see also \cite[sec 10.1.4]{P} then guarantees the sequence $\{(\bar{B}_{l\cdot\frac{\delta}{2}(p_k)},g_k,p_k)\}_{k=1}^{\infty}$ converges(after passing to a subsequence) to a metric space $(\bar{B}_{l\cdot\frac{\delta}{2}}(p),d_l,p)$ in pointed Gromov-Haustorff topology. W.L.O.G., one could assume $(\bar{B}_{l\cdot\frac{\delta}{2}}(p),d_l,p)\subset (\bar{B}_{{l+1}\cdot\frac{\delta}{2}}(p),d_{l+1},p)$. After taking direct limit, we have
  $$(\Sigma_k,g_k,p_k)=\lim_{\longrightarrow}\bar{B}_{l\cdot\frac{\delta}{2}}(p_k)\xrightarrow[]{p-GH}
  \lim_{\longrightarrow}\bar{B}_{l\cdot\frac{\delta}{2}}(p)=:(X,d,p).$$

  From now on, we assume all $(\Sigma_k,g_k,p_k)$ and $(\Sigma,d,p)$ are in a same metric space $Y$ locally since Gromov-Haustorff convergence is equal to Haustorff convergence after passing to a subsequence.\\
  Step 2.(Convergence of complex structure)

  Assume $f_{ks}:D_1(0)\to U_{ks}(p^k_s)(\hookrightarrow Y)$ are the isothermal coordinates taken above and $e^{2u_{ks}}g_0=f_{ks}^*g_k=\langle df_{ks},df_{ks}\rangle=({|\frac{\partial f_{ks}}{\partial x}|}^2+{|\frac{\partial f_{ks}}{\partial y}|}^2)g_0$. Then, Lemma~\ref{L^infty} implies $\|\nabla f_{ks}\|_{L^{\infty}(D_r)}\le e^{\|u_{ks}\|_{L^{\infty}(D_r)}}\le e^{C(r)}$, i.e., $\{f_{ks}\}_{k=1}^{\infty}$ are all local Lipschitz  with uniform Lipschitz constant. So Arzela-Ascoli's lemma implies (after passing to a subsequence) $f_{ks}\xrightarrow[]{C^{0,\beta}(D_r)}f_s:D_r(0)\to Y$. Furthermore, we have $$d(f_s(x),f_s(y))=\lim_{k\to\infty}d(f_{ks}(x),f_{ks}(y))\le e^{C(r)}|x-y|,$$
   and
  $$d(f_s(x),f_s(y))\ge e^{-C(r)}\mathop{\mathrm{min}}\{|x-y|,2r-|x|-|y|\}\ge e^{-C(r)}\frac{r-\sigma}{\sigma}|x-y|,$$ for $x,y\in D_{\sigma}(0)\subset\subset D_r(0)$, i.e., $f_s$ is local bilipschitz with $\mathrm{Lip}_{D_r}f_s\le e^{C(r)}$ and hence also injective. If we define $p_s=\lim_{k\to \infty}p_{ks}$, then when noticing that all $X_k$ are length spaces, one know $U_{ks}(p_s^k)$ converges to some $U_s(p_s)\supset B^X_{i_0}(p_s)\subset X$ in Haustorff topology as subsets in $Y$. So $f_{ks}\xrightarrow[]{C^0}f_s:\bar{D}_r(0)\to Y$ implies $\mathrm{Im}(f_s)\subset X$, i.e., $f_s:(D_1(0),0)\to(U_s(p_s))\subset X$ is an injective map from a compact space to a Haustorff space, hence is an embedding.

  Moreover, we claim $f_s:D_1(0)\to U_s(p_s)$ is surjective (hence homeomorphic). In fact, for $\forall x\in U_s(p_s)=\cup_r\lim_{k\to\infty}f_{ks}(D_r)$ (Haustorff convergence as subsets in $Y$), $\exists r\in(0,1)$ and $x_k=:f_{ks}(a_k)\in f_{ks}(D_r(0))$ s.t. $x_k\to x$. W.L.O.G., we can assume $a_k\to a\in \bar{D}_r(0)$. Then
  \begin{align*}
  0\le d(f_s(a),x)&\le d(f_s(a_k),f_s(a))+d(f_s(a_k),x_k)+d(x_k,x)\\
  &\le e^{C(r)}|a_k-a|+d(f_s(a_k),f_{ks}(a_k))+d(x_k,x)\to 0,
   \end{align*}
   since $f_{ks}\to f_s$ uniformly on compact subsets of $D_1(0)$. So, $x=f_s(a)$ and $f_s$ is surjective. This means $\Sigma:=X$ is a topological manifold.

  To construct a complex structure on $X$, we consider the transport function $f_{ks}^{-1}\circ f_{kt}$ with domain $\mathrm{Dom}(f_{ks}^{-1}\circ f_{kt})\to\mathrm{Dom}(f_{s}^{-1}\circ f_{t})$ in Haustorff topology as subsets in $\mathbb{C}$. We have $d(f_{ks}^{-1}\circ f_{kt}(z),f_{s}^{-1}\circ f_{t}(z))\le d(f_{ks}^{-1}\circ f_{kt}(z),f_{ks}^{-1}\circ f_{t}(z))+d(f_{ks}^{-1}\circ f_{t}(z),f_{s}^{-1}\circ f_{t}(z))\to 0$ uniformly on compact subsets of $\mathrm{Dom}(f_{s}^{-1}\circ f_{t})$ since $f_{kt}\xrightarrow[]{C_c^0}f_t$ and $f_{ks}^{-1}$ are locally uniformly bilipschitz. This means
  $$f_{ks}^{-1}\circ f_{kt}\xrightarrow[]{C_c^0} f_s^{-1}\circ f_t\  \mathrm{on} \ \mathrm{Dom}(f_s^{-1}\circ f_t).$$
  But we know $f_{ks}^{-1}\circ f_{kt}$ is analytic since $f_{ks}$ and $f_{kt}$ are both conformal coordinates in their intersection domain, so Montel's theorem implies $f_s^{-1}\circ f_t$ is also analytic on its domain. That means $\mathcal{O}=\{f_s:D_1(0)\}_{s=1}^{\infty}$ is a complex structure on $\Sigma$ and $$(\Sigma_k,\mathcal{O}_k,p_k)\xrightarrow[]{C^{\omega}}(\Sigma,\mathcal{O},p),$$
  as pointed Riemannian surface, where $\mathcal{O}_k=\{f_{ks}:D_r(0)\to U_{ks}(p^k_s)\subset\Sigma_k\}_{s=1}^{\infty}$.\\
  Step 3.(Riemannian metric on $\Sigma$)

   For the isothermal coordinate $f_{ks}:D_1\to U_{ks}\subset\Sigma_k\to\mathbb{R}^n$ with pull back metric represented as $f_{ks}^*g_k=e^{2u_{ks}}g_0$. Recall $-\triangle u_{ks}=K_{k}e^{2u_{ks}}=w_{ks}$ in $D_1$ for some $w_{ks}\in\mathcal{H}^1(\mathbb{C})$. As before, we define $v_{ks}=-{\triangle}^{-1}w_{ks}$ and $h_{ks}=u_{ks}-v_{ks}$. Then, $(\Sigma_k\to\mathbb{R}^n)\in \mathcal{E}$ and Lemma~\ref{extend} implies $$\|w_{ks}\|_{\mathcal{H}^1(\mathbb{C})}\le C(n,\varepsilon,A_0,i_0).$$
   Furthermore, by Lemma~\ref{extend}, we have
   $$\|v_{ks}\|_{L^{\infty}(\mathbb{C})}+\|v_{ks}\|_{W^{1,2}(\mathbb{C})}\le C\|w_{ks}\|_{\mathcal{H}^1(\mathbb{C})}\le C(n,\varepsilon,A_0,i_0).$$
   Now, Rellich's lemma and the weak compactness of $\mathcal{H}^1(\mathbb{C})$(see \cite[chap 3.5.1]{S}) imply there exist $w_s\in \mathcal{H}^1(\mathbb{C})$ and $v_s\in W^{1,2}(\mathbb{C})$ s.t.
   $$w_{ks}\rightharpoonup w_s \ \mathrm{as\ distribution}, \ and\ v_{ks}\xrightarrow[]{L^p(\mathrm{hence}\ a.e.)}v_s,\ \forall p\in(1,+\infty).$$
   Moreover, for any $\varphi\in C_c^{\infty}(\mathbb{C})$, we have
   $$\int_{\mathbb{C}}\nabla v_s \nabla \varphi=\lim_{k\to \infty}\int_{\mathbb{C}}\nabla v_{ks} \nabla \varphi=\lim_{k\to\infty}\int_{\mathbb{C}}w_{ks}\varphi=\int_{\mathbb{C}}w_s\varphi,$$
   i.e., $v_s\in W_0^{1,2}(\mathbb{C})$ satisfies the weak equation $-\triangle v_s=w_s$ in $\mathbb{C}$. And weak lower semi-continuity of the norm of the Banach space $\mathcal{H}^1=(VMO)^*$ imply $\|w_s\|_{\mathcal{H}^1(\mathbb{C})}\le\liminf_{k\to\infty}\|w_{ks}\|_{\mathcal{H}^1(\mathbb{C})}\le C(n,\varepsilon, A_0,i_0)$. So, by Lemma~\ref{extend} again, we get $v_s\in C^0(\mathbb{C})$ and
   $$\|v_s\|_{L^{\infty}(\mathbb{C})}+\|v_s\|_{W^{1,2}(\mathbb{C})}\le C \|w_s\|_{\mathcal{H}^1(\mathbb{C})}\le C(n,\varepsilon, A_0,i_0).$$
   On the other hand, by Corollary~\ref{estimate for h}, there exists $h_s$ harmonic on $D_1$ such that $h_{ks}\xrightarrow[]{C_c^{\infty}(D_1)}h_s$. Denote $u_s:=h_s+v_s$. Then we get $u_{ks}\xrightarrow[]{L^p_{loc}(D_1)}u_s$ and $u_s\in C^0(D_1)$ with $\|u_s\|_{L^{\infty}(D_r)}\le C(n,\varepsilon,A_0,i_0,V, r), \forall r \in(0,1)$.

   With this $u_s$, we can construct a continuous local metric $g$ on $\Sigma$ by defining
   $$g={(f_s^{-1})}^*(e^{2u_s}g_0),$$
   where $f_s:D_1\to U_s(p_s)$ is a coordinate in the complex structure $\mathcal{O}$ constructed in Step 2.  In fact, Step 2. also claim $\mathcal{O}_k\to\mathcal{O}$, i.e., $f_{ks}^{-1}\circ f_{kt}\xrightarrow[]{C_c^{\infty}}f_s^{-1}\circ f_t$, from which we can get ${(f_s^{-1})}^*(e^{2u_s}g_0)={(f_t^{-1})}^*(e^{2u_t}g_0)$ a.e.(hence everywhere since both are continuous) on their common domain, i.e., the metric $g$ is globally well defined. Moreover, $u_{ks}\to u_s\ a.e.$ on $D_1$ and $\|u_s\|_{L^{\infty}(D_r)}\le C(n,\varepsilon,A_0,i_0,V, r)$ imply $\lim_{k\to \infty}l_{g_k}(\gamma)=l_g(\gamma), \forall \gamma \in C^1([0,1],\Sigma)$ by dominate convergence theorem. Thus the induced metric $d_g$ is compatible with the limit metric $d$. More precisely, we have
    $$d_g(x,y)=\inf_{\gamma} l_g(\gamma)=\inf_{\gamma}\lim_{k\to\infty}l_{g_k}(\gamma)\ge \lim d_{g_k}(x,y)=d(x,y), \forall x,y \in \Sigma,$$
    where we use the Gromov-Haustorff convergence result from Step 1. in the last equation. On the other hand, since $\|u_{ks}\|_{L^{\infty}(D_r)}+\|u_s\|_{L^{\infty}(D_r)}\le C(n,\varepsilon,A_0,i_0,V, r)$, one know all the metrics are uniformly equivalent to the Euclidean metric in every coordinate and $$d_g\le C d$$ for some $C=C(n, \varepsilon, A_0, i_0,V)$.

    Moreover, in the case $\int_{\Sigma_k}|A_k|^2d\mu_{g_k}\to 0$, by Corollary \ref{uniformly convergence}, we know $u_{ks}$ converges to $u_s$ in $C^0_{loc}(D_1)$.So for any $\delta>0$, $u(x)-\delta\le u_k(x)\le u(x)+\delta$ for large $k$ independent on $x\in D_{r}(0)$. For any $\gamma:[0,1]\to D_r(0)$,
    $$\int_0^1 e^{-\delta}e^{u\circ \gamma}|\gamma'|\le \int_0^1 e^{u_k\circ \gamma}|\gamma'|\le\int_0^1 e^{\delta}e^{u\circ \gamma}|\gamma'|.$$
     Since the last estimate is uniform for all $\gamma$, we can take infimum for $\gamma$ joining $x$ and $y$ and then let $k\to \infty$ and $\delta\to 0$ to get
    $$d_{g}(x,y)=\lim_{k\to \infty}d_{g_k}(x,y)=d(x,y).$$
   Step 4.($L^p$-convergence of the metric structure)

   By gluing the local diffeomorphisms $\phi_{ks}=f_{ks}\circ f_s^{-1}:U_s\to D_1\to U_{ks}$ which converge to identity together by partition of unity, we get the following global description of pointed-convergence of differential structure.
   \begin{lemma}\label{gluing lemma}
   If $(\Sigma_k, \mathcal{O}_k,p_k)$ converges to $(\Sigma, \mathcal{O},p)$ as pointed differential surfaces, then for any fixed $l$, there exist differential maps $\Phi_{kl}:{\Omega}_l=\cup_{s=1}^{l}U_s\to{\Omega}_{kl}=\cup_{s=1}^l U_{ks}$ for $k$ large enough, such that $\Phi_{kl}$ are embeddings  when restricted to compact subsets of $\Omega_l$.
   \end{lemma}
   \begin{proof}(It can be found in \cite[sec 10.3.4]{P}, we write it here for the convenience of readers.)
   Assume $\mathcal{O}_k=\{f_{ks}:D_1\to U_{ks}\}_{s=1}^{\infty}$ and $\mathcal{O}=\{f_s:D_1\to U_s\}_{s=1}^{\infty}$. Define $\phi_{ks}:f_{ks}\circ f_s^{-1}:U_s\to D_1\to U_{ks}$. Then for $t\neq s$, if $U_s \cap U_t\neq \varnothing$, when putting $\phi_{kt}:U_t\to U_{kt}$ in local coordinates $f_s:D_1\to U_s$ and $f_{ks}:D_1\to U_{ks}$ of $\Sigma$ and $\Sigma_k$, we have
   $$\tilde{\phi}_{kt}=f_{ks}^{-1}\circ \phi_{kt}\circ f_s=f_{ks}^{-1}\circ f_{kt}\circ f_t^{-1}\circ f_s\xrightarrow[]{C_c^{\infty}}f_s^{-1}\circ f_t \circ f_t^{-1} \circ f_s=id.$$
   That is, the local map $\phi_{kt}$ between $\Sigma$ and $\Sigma_k$ converges smoothly to identity  w.r.t. the differential structures $\mathcal{O}$ and $\mathcal{O}_k$. That is, if we denote $\hat{\phi}_{ks}=f_{ks}^{-1}\circ \phi_{ks}$, then for any compact subset $ K\subset \mathrm{Dom}(\phi_{ks})\cap \mathrm{Dom}(\phi_{kt})=U_s\cap U_t$ and integer $m$, $$\|\hat{\phi}_{kt}-\hat{\phi}_{ks}\|_{C^{m}(K)}\le\|\hat{\phi}_{kt}-id\|_{C^{m}(K)}+\|id-\hat{\phi}_{ks}\|_{C^{m}(K)} \to 0, \text{ as }k\to +\infty.$$
   Now, choose a partition of unity $\{\lambda_1, \lambda_2\}$ for $\{U_s, U_t\}$, i.e., smooth functions $\lambda_1,\ \lambda_2$  on $\Sigma$ with $\mathrm{supp}\lambda_1\subset U_s$, $\mathrm{supp}\lambda_2\subset U_t$ and $\lambda_1+\lambda_2=1$ on $U_s\cup U_t$. Then $\lambda_1=1$ on $U_s\backslash U_t$ and $\lambda_2=1$ on $U_t \backslash U_s$. Let $\hat{\Phi}_k=\lambda_1\hat{\phi}_{ks}+\lambda_2\hat{\phi}_{kt}$. Then $\hat{\Phi}_k-\hat{\phi}_{ks}=(\lambda_1-1)\hat{\phi}_{ks}+\lambda_2\hat{\phi}_{kt}=\lambda_2(\hat{\phi}_{kt}-\hat{\phi}_{ks})
   \to 0$ in $C_c^{\infty}(U_s)$. For the same reason, $\hat{\Phi}_k-\hat{\phi}_{kt}\to 0$ in $C_c^{\infty}(U_t)$. But we know  $\hat{\phi}_{ks}\to id$ in $C_c^{\infty}(U_s)$ and $\hat{\phi}_{kt}\to id$ in $C_c^{\infty}(U_t)$, so if we define
   $$
   \Phi_k=
   \begin{cases}
   f_{ks}\circ \hat{\phi}_{ks}\ \mathrm{on}\  U_s,\\
   f_{kt}\circ \hat{\Phi}_k\ \mathrm{on}\  U_t\backslash U_s,\\
   \end{cases}
   $$
   then $\Phi_k$ is well defined and for any compact subset $K\subset\subset U_s\cup U_t$, $\Phi_k:K\to \Phi_k(K)\subset U_{ks}\cup U_{kt}$ is a diffeomorphism for $k$ large enough.

   The above argument show that we can glue two sequences of diffeomorphisms together if they are arbitrary close for $k$ large enough. So, by induction, for the finite sequences of local coordinates $\{\phi_{ks}\}_{s=1}^l:U_s\to U_{ks}$, they are all close to the identity(hence to each other) for $k$ large enough. When denoting $\Omega_l=\cup_{s=1}^l U_s$ and $\Omega_{kl}=\cup_{s=1}^l U_{ks}$, we can glue them together to be a local diffeomorphism $\Phi_{kl}:\Omega_l\to \Omega_{kl}$ such that $\Phi_{kl}^{-1}\circ \phi_{ks}$ converges smoothly to the identity on its domain.
   \end{proof}
   By this lemma, we have ${(\Phi_{kl})}^*g_k\xrightarrow[]{L^p_{loc}}g$ since $(\phi_{ks})^*g_k\xrightarrow[]{L^p}g$ by the construction of $g$ and $\Phi_{kl}$ is arbitrary close to $\phi_{ks}$ in $U_s$ for $k$ large enough. That is,
   $$(\Sigma_k,g_k,p_k)\to(\Sigma,g,p)\sim (X,d,p)\ \ \mathrm{in\ pointed}-L^p\ \mathrm{topology}.$$
   In the case $\int_{\Sigma_k}|A_k|^2d\mu_{g_k}\to 0$, by gluing the local uniform convergence of $u_{ks}\to u_{s}$, we know $(\Sigma_k,g_k,p_k)$ converges to $(\Sigma,g,p)$ in $C^0_{loc}$ topology.
  \end{proof}
  \subsection{Extrinsic convergence}
  Assume $\{F_k:(\Sigma_k,g_k,p_k)\to\mathbb{R}^n\}^{+\infty}_{k=1}$ converges to $(\Sigma,g,p)$ in intrinsic pointed-$L^p$-topology. Then there exist $\Phi_k:(\Sigma,p)\to (\Sigma_k,p_k)$ such that $\Phi_k^*g_k\to g$ in $L^p_{loc}$. Moreover, one could use $X_k:=F_k\circ \Phi_k$(they have common domain $\Sigma$) to represent the immersion $F_k$ and ask whether these $X_k$ converge in $W^{k,p}(\Sigma, \mathbb{R}^n)$. In this sense, we have the following extrinsic convergence theorem.
  \begin{thm}
  Assume $\{F_k:(\Sigma_k,g_k,p_k)\to\mathbb{R}^n\}^{+\infty}_{k=1}\subset \mathcal{E}(n,\varepsilon,A_0,i_0,V)$ is the sequence converges to $(\Sigma,g,p)$ in pointed-$L^p$-topology as in theorem~\ref{intrinsic convergence} with $F_k(p_k)\equiv 0$. Then there exists a isometric immersion $F:\Sigma\to \mathbb{R}^n$ such that $F_k$ converges to $F$ weakly in $W^{2,2}_{loc}(\Sigma,\mathbb{R}^n)$ and strongly in $W^{1,p}_{loc}(\Sigma,\mathbb{R}^n)$.
  \end{thm}
  \begin{proof}\

  Step 1.(local convergence)
  Take notations as in the proof of the intrinsic convergence: $i_0=$ the common isothermal radius, $f_{ks}:D_1\to U_{ks}(p^k_{s})\subset \Sigma_k$ and $f_s:D_1\to U_s(p_s)\subset\Sigma$ local isothermal coordinates with metrics $f_{ks}^*g_k=e^{2u_{ks}}g_0$ and $f_s^*g=e^{2u_s}g_0$ respectively, $\phi_{ks}=f_{ks}\circ f_s^{-1}:U_s(p_s)\to U_{ks}(p^k_{s})$ diffeomorphism such that $\tilde{g}_k=\phi_{ks}^*g_k\xrightarrow[]{L^p}g$. Assume $F_k:(\Sigma_k,p_{k1})\to \mathbb{R}^n$ to be the immersion of $\Sigma_k$ into $\mathbb{R}^n$ such that $F_k(p_{k1})=0$ and regard $F_{ks}=F_k\circ \phi_{ks}:U_s(p_s)\to U_{ks}(p^k_{s})\to\mathbb{R}^n$ as immersion(not necessarily isometry) of $U_s(p_s)$ into $\mathbb{R}^n$. Then under the isothermal coordinate $f_s:D_1\to U_s(p_s)$, if we define $\tilde{F}_{ks}=F_{ks}\circ f_s:D_1\to\mathbb{R}^n$, then
  $$\triangle \tilde{F}_{ks}=e^{2u_{ks}}\triangle_{g_k}\tilde{F}_{ks}=e^{2u_{ks}}\vec{H}_k=:h_{ks},$$
  where we use the mean curvature equation in the last equation. Note that
  $$\int_{D_r}|h_{ks}|^2=\int_{D_r}e^{4u_{ks}}|\vec{H}_{ks}|^2\le e^{2C(n,\varepsilon, A_0,V,r)}\int_{\Sigma_k}|A_k|^2\le e^C A_0$$
  and
  $$|\tilde{F}_{k1}|(x)\le|\tilde{F}_{k1}(x)-\tilde{F}_{k1}(0)|\le|d\tilde{F}_{k1}||x|\le e^Cr.$$
  By local $L^2$-estimation for elliptic equation, we get
  $$\|\tilde{F}_{k1}\|_{W^{2,2}(D_{r'})}\le C(r,r')(\|\tilde{F}_{k1}\|_{L^2(D_r)}+\|h_{ks}\|_{L^2(D_r)})\le C(r)<\infty.$$
  Thus, there exists an $\tilde{F}_{\infty1}\in W^{2,2}_{loc}(D_1,\mathbb{R}^n)$ such that $\tilde{F}_{k1}$ converges to $\tilde{F}_{\infty1}$ weakly in $W^{2,2}_{loc}(D_1,\mathbb{R}^n)$ and strongly in $W^{1,p}_{loc}(D_1,\mathbb{R}^n)$. Hence $\tilde{F}_{k1}^*g_{\mathbb{R}^n}$ converges to $\tilde{F}_{\infty1}^*g_{\mathbb{R}^n}$ in $L^p_{loc}(D_1)$. But we know
  \begin{align*}
  \tilde{F}_{k1}^*g_{\mathbb{R}^n}=(F_k\circ\phi_{k1}\circ f_1)^*g_{\mathbb{R}^n}=f_{k1}^*g_k\to f_1^*g\ \mathrm{in}\ L^p_{loc}(D_1).
  \end{align*}
  So, $\tilde{F}_{\infty1}^*g_{\mathbb{R}^n}=f_1^*g$, i.e., $F_{\infty1}:=\tilde{F}_{\infty1}\circ f_1^{-1}:U_1(p_1)\to\mathbb{R}^n$ is a local $Riemannian$ immersion.

  Step 2. (extending the local limit to global) We argue by induction. Let $\Omega_1=U_1(p_1)$, $\Omega_{l+1}=\Omega_l\cup \cup_{p_s\in \Omega_l}U_s(p_s)$; $\Phi_{k1}=\phi_{k1}, \Phi_{kl}:\Omega_l\to\Omega_{kl}\subset\Sigma_k$ constructed in Lemma~\ref{gluing lemma}. Now, assume we have constructed isomorphism immersion $F_{\infty l}:\Omega_l\to\mathbb{R}^n$ such that $X_{kl}:=F_k\circ\Phi_{kl}$ converges to $F_{\infty l}$ weakly in $W^{2,2}_{loc}(\Omega_l,\mathbb{R}^n)$ and strongly in $W^{1,p}_{loc}(\Omega_l,\mathbb{R}^n)\cap C^{\beta}_c(\Omega_l,\mathbb{R}^n)$. Then for any $p_s\in \Omega_l$, $F_k(p_s)\to F_{\infty l}(p_s)$. Noticing $|\tilde{F}_{ks}|\le|\tilde{F}_{ks}(p_s)|+e^Cr\le C(r)$ in $D_r$ and repeating the progress in Step 1, we know $F_{ks}=F_k\circ \phi_{ks}(\approx F_k\circ \Phi_{k(l+1)})$ converges weakly in $W^{2,2}_{loc}(U_s(p_s),\mathbb{R}^n)$ and strongly in $W^{1,p}_{loc}(U_s,\mathbb{R}^n)\cap C^{\beta}_c(U_s(p_s),\mathbb{R}^n)$. Moreover, by construction of $\Phi_{kl}$, we have
  $$X_{kl}^{-1}\circ F_{ks}=\Phi_{kl}^{-1}\circ \phi_{ks}\to id\ \mathrm{in}\  C_c^{\infty}(\Omega_l\cap B_{i_0}(p_s)).$$
  Thus $F_{ks}$ and $X_{kl}$ converges to the same limit in their common domain, i.e., $X_{k(l+1)}:=F_k\circ\Phi_{k(l+1)}$ converges to some $F_{\infty l+1}$ weakly in $W^{2,2}_{loc}(\Omega_{l+1},\mathbb{R}^n)$ and strongly in $W^{1,p}_{loc}(\Omega_{l+1},\mathbb{R}^n)\cap C^{\beta}_c(\Omega_{l+1},\mathbb{R}^n)$, and $F_{\infty l+1}:\Omega_{l+1}\to\mathbb{R}^n$ is an isometric immersion which extends $F_{\infty l}$. The work is done by taking  a direct limit $F=\underrightarrow{\lim} F_{\infty l}$.
  \end{proof}

  \begin{cor}
  By Langer's weak lower semi-continuous theorem\cite{L}, we also know $$\int_{\Sigma}{|A|}^2\le \liminf_{k\to \infty}\int_{\Sigma_k}{|A_k|}^2\le A_0$$ under the above assumption.
  \end{cor}
  \begin{proof}
  The result can be found in \cite{L}, \cite{Mo}, we write the proof here for the convenience of the reader.
  The key observation is, as a function of $(\xi,\zeta,\eta)=(F,DF,D^2F)$, $J(\xi,\zeta,\eta):=|A_g|_g^2d\mu_g$ depends on $\eta=D^2F$ convexly for each fixed $\xi$ and $\zeta$, since it is an nonnegative quadratic form of $D^2F$ in any fixed coordinate.  For any compact domain $D\subset\subset \Sigma$, we suppose $\int_{D}|A|^2_gd\mu_g<\infty$. Since $F_k$ converges to $F$ weakly in $W^{2,2}_{loc}(\Sigma,\mathbb{R}^n)$ and strongly in $W^{1,p}_{loc}(\Sigma,\mathbb{R}^n)$, by Lusin's theorem, Egorov's theorem and the absolute continuity of integral, there exists a compact set $S\subset\subset D$ such that $F, DF, D^2F$ are continuous on $S$, $(F_k,DF_k)$ converges to $(F,DF)$ uniformly on $S$ and
  $$\int_{S}|A|^2_gd\mu_g\ge\int_{D}|A|^2_gd\mu_g-\varepsilon.$$
  In the case $\int_{D}|A|^2_gd\mu_g=\infty$, we can take $\int_{S}|A|^2_gd\mu_g\ge M$ for any $M>0$.  By  the convexity of $J$ we know
  \begin{align*}
  J(F_k,DF_k, D^2F_k)\ge J(F_k,DF_k,D^2F)&+D_{\eta}J(F,DF,D^2F)\cdot(D^2F_k-D^2F)\\
  &+[D_{\eta}J(F_k,DF_k,D^2F)-D_{\eta}J(F,DF,D^2F)]\cdot(D^2F_k-D^2F).
  \end{align*}
  Since $F,DF,D^2F\in C(S)$ and $dF\otimes dF$ is a metric, we know $D_{\eta}J(F,DF,D^2F)\in C(S)$. By the weak $L^2$ convergence of $D^2F_k$ to $D^2F$, we get
  \begin{align*}
  \int_{S}D_{\eta}J(F,DF,D^2F)\cdot(D^2F_k-D^2F)\to0.
  \end{align*}
  Moreover, the uniform convergence of $(F_k,DF_k)$  to $(F,DF)$ on $S$ and uniform boundness of $D^2F_k$ and $D^2F$ in $L^2$  imply that $D_{\eta}J(F_k,DF_k,D^2F)$ converges to $D_{\eta}J(F,DF,D^2F)$ uniformly on $S$ and
  \begin{align*}
  \int_{S}[D_{\eta}J(F_k,DF_k,D^2F)-D_{\eta}J(F,DF,D^2F)]\cdot(D^2F_k-D^2F)\to 0.
  \end{align*}
  As a result, we get
  \begin{align*}
  \liminf_{k\to \infty}\int_SJ(F_k,DF_k, D^2F_k)&\ge \liminf_{k\to \infty}\int_SJ(F_k,DF_k, D^2F)\\
  &=\int_SJ(F,DF, D^2F)\\
  &\ge \int_DJ(F,DF, D^2F)-\varepsilon.
  \end{align*}
  Letting $\varepsilon\to 0$ and $D\to \Sigma$, we get
  $$\int_{\Sigma}{|A|}^2\le \liminf_{k\to \infty}\int_{\Sigma_k}{|A_k|}^2.$$
  \end{proof}
  \subsection{Proof of Theorem \ref{convergence}}
  In this subsection, we will add the above two subsections to prove Theorem \ref{convergence}. The key work we need to do in the local case is to use the bounded volume condition to control the limit mapping such that it is proper.
  \begin{proof}
  Recall we use $\Sigma_k^{r}(p_k)$ to denote the connected component of $F_k^{-1}(B_r(0))\cap \Sigma_k$ containing $p_k$ and use $B_r^{\Sigma_k}(p_k)$ to denote the open geodesic ball in $\Sigma_k$ centered at $p_k$ with radius $r$. Then for each $r\in (0,R]$, $ B_r^{\Sigma_k}(p_k)\subset\Sigma_k^r(p_k)$. Since $i_{g_k}(x)\ge i_0(R-|F_k(x)|)$, by the same argument as in the global case(the above two subsections), we know there exists a (non-complete) pointed Riemannian surface $(\Sigma_r,g,d,p)$ with continuous metric $g$ such that $d_g\sim d$ and an isometric immersion $F_r:(\Sigma_r,p)\to (\mathbb{R}^n,0)$ such that,  after passing to a subsequence,
   \begin{enumerate}[$1)$]
  \item   $(B_{r}^{\Sigma_k}(p_k), g_k,p_k)$ converges  to $(\Sigma_r,g,p)$ in pointed $L^p$ topology, i.e., there exist smooth embedings $\Phi_{kr}:\Sigma_r\to B_{r}^{\Sigma_k}(p_k)$, s.t. $\Phi_{kr}^{*}g_k\to g$ in $L^p_{loc}(d\mu_g)$.
  \item The complex structure $\mathcal{O}_{kr}$ of $B_{r}^{\Sigma_k}(p_k)$ converges to the complex structure $\mathcal{O}$ of $\Sigma_r$.
  \item  $F_k\circ \Phi_{kr}$ converges to $F_r$ weakly in $W^{2,2}_{loc}(\Sigma_r,\mathbb{R}^n)$ and strongly in $W^{1,p}_{loc}(\Sigma_r,\mathbb{R}^n)$.
  \end{enumerate}

  But we do not know whether $F_R$ is proper. We even do not know whether  the `boundary' of $F_R(\Sigma_R)$(which we will define latter) will touch $\partial B_R(0)$. So we will next extend the local limit $X_R:=(F_R,\Sigma_R,g,p)$ to some maximal $X=(F,\Sigma,g,p)$ such that the `boundary' of $F(\Sigma)$  touches $\partial B_R(0)$  and then argue the extended immersing map $F$ is proper in $B^R(0)$. For simplification, we call the topology defined by the convergence in the last three items $1)\ 2)\ 3)$ by $\tau$-topology. For example, we will say the quadruple $X_{kr}=(F_k,B_{r}^{\Sigma_k}(p_k), g_k,p_k)$ converges to $X_r=(F_r,\Sigma_r,g,p)$ in $\tau$-topology and so on. For $\forall r\in(0,R)$, we define a partial order on the set
  \begin{align*}
  \mathcal{C}_r=\{X^A&=(F^A,\Sigma^A,g,p)| \exists A\subset \mathbb{N}_+\text{ and connected domains } U^A_{\alpha}\subset \Sigma_{\alpha}^r(p_{\alpha}),\forall  \alpha\in A\text{ s.t.}\\
                     &U^A_{\alpha}\supset B_r^{\Sigma_{\alpha}}(p_{\alpha}) \text{ and } X_{\alpha}=(F_{\alpha},U^A_{\alpha},g_{\alpha},p_{\alpha})\text{ converges to }X^{A}\text{ in } \tau\text{-topology } \}
 \end{align*}
 by saying $X^{A}\le X^B$ if $A\cap B$ is infinite and $U^A_{\gamma}\subset U^B_{\gamma}$ for infinite $\gamma\in A\cap B$. By diagonal argument, it is not hard to show that $(\mathcal{C}_r,\le)$ is a partial order set such that every chain in $(\mathcal{C}_r,\le)$ has an upper bound. Thus by Zorn's lemma, there exists a maximal $X^{r}=(F^{r},\Sigma^{r},g,p)=$ some $X^{A_{r}}$ in $(\mathcal{C}_r,\le)$. We are going to show $F^{r}:\Sigma^{r}\to B_r(0)$ is proper for each $r\in (0,R)$.

 First of all, we define $\bar{\Sigma}^r$ to be the completion of $(\Sigma^r,d_g)$ as a metric space and call $\partial \Sigma^r:=\bar{\Sigma}^r\backslash \Sigma^r$ the boundary of $\Sigma^r$. Since the area of a surface is continuous in the $\tau$-topology, we know $vol_g(\Sigma^r)\le V<+\infty$. Hence $\Sigma^r$ is not complete and $\partial\Sigma^r\neq \emptyset$.

 Otherwise, $\Sigma^r$ will be a closed surface or contains a ray. In the first case, $U_{\alpha}$ is a closed surface for $\alpha\in A_r$ large enough. This means $\Sigma^{R}_{\alpha}(p_{\alpha})$ contains a closed surface $U_{\alpha}$ and must equals to the closed surface  itself since it is connected. But then $\Sigma^{R}_{\alpha}(p_{\alpha})=U_{\alpha}\subset B_r(0)\subset\subset B_R(0)$, which contradicts to the fact $F_{\alpha}:\Sigma_{\alpha}\to B_R(0)$ is a proper immersion of an open surface. In the later case, we can choose infinite many points $\{x_i\}$ on the ray such that $B^{\Sigma^r}_{i_0(R-r)}(x_i)\cap B^{\Sigma^r}_{i_0(R-r)}(x_j)=\emptyset$ for $\forall i\neq j$. But by the definition of the isothermal radius, there exist open neighborhoods $U(x_i)\supset B^{\Sigma}_{i_0(R-r)}(x_i)$ and isothermal coordinates $\varphi_i:D_1(0)\to U(x_i) $ such that $\varphi_i(0)=x_i$ and $\varphi_i^*g=e^{2u}(dx^2+dy^2)$. Moreover, we know $|u|(z)\le C(n,\varepsilon,V,i_0(R-r))$ for $|z|\le \frac{1}{2}$. Thus for $|z|\le r_0:=e^{-C(n,\varepsilon, V,i_0(R-r))}\frac{i_0(R-r)}{2}$,
 \begin{align*}
 d_g(z,0)\le \int_0^1e^{u(tz)}dt\le\frac{i_0(R-r)}{2}.
 \end{align*}
This means $B^{\Sigma^r}_{i_0(R-r)}(x_i)\supset \varphi_i(D_{r_0}(0))$ and
\begin{align*}
vol_g(B^{\Sigma^r}_{i_0(R-r)}(x_i))\ge \int_{D_{r_1}(0)}e^{2u}dxdy\ge \frac{\pi i^2_0(R-r)}{4}e^{-4C(n,\varepsilon, V,i_0(R-r))}.
\end{align*}
So we get
$$V\ge vol_g(\Sigma^r)\ge \sum_{i=1}^{\infty}vol_g(B^{\Sigma^r}_{i_0(R-r)}(x_i))=+\infty,$$
again a contradiction!

 Then we claim $\bar{\Sigma}^r_{\delta}:=\{x\in \Sigma^r|d_g(x,\partial \Sigma^r)\ge\delta\}$ is compact in $\Sigma^r$. Here we also need to use the area bound  essentially. In fact, if we choose a maximal disjoint family of closed balls $\{\bar{B}^g_{\frac{\delta}{3}}(x_l)\}_{\l\in I}$ s.t. $x_l\in \bar{\Sigma}^r_{\delta}$, then $\bar{\Sigma}^r_{\delta}\subset \cup_{l\in I}\bar{B}^g_{\frac{2\delta}{3}}(x_l)$. So if the number of the balls is finite, then $\bar{\Sigma}^r_{\delta}$ is compact. Since $X_{\alpha}^{A_r}=(F_{\alpha},U_{\alpha}^{A_r}),g_{\alpha},p_{\alpha})$ converges to $X^r=(F^r,\Sigma^r,g,p)$ in $\tau$-topology, there exists $\Phi_{\alpha}^{A_r}:\Sigma^r\to U_{\alpha}^{A_r}$, s.t. $\Phi_{\alpha}^{A_r*}g_{\alpha}\to g$ in $L^p_{loc}$-topology and $x_{\alpha l}=\Phi_{\alpha}^{A_r}(x_l)\to x_l$ in pointed Gromov-Haustorff topology. Since $d_g\le Cd$, we know $\Phi_{\alpha}^{A_r}(\bar{B}^g_{\frac{\delta}{3}}(x_l))\supset \bar{B}^{g_{\alpha}}_{\frac{\delta}{6C}}(x_{\alpha l})$ for $\alpha$ large enough. And by the same argument as in the above paragraph, we get
\begin{align*}
vol_g(\bar{B}^g_{\frac{\delta}{3}}(x_l))&\ge \lim_{\alpha\to \infty}vol_{g_k}(\bar{B}^{g_{\alpha}}_{\frac{\delta}{6C}}(x_{\alpha l}))\\
                                        &\ge \lim_{\alpha\to\infty}\pi\min\{\big(\frac{\delta}{12C}\big)^2,\frac{i_0(R-r))}{2}^2\}e^{-4C(n,V,i_0(R-r),\varepsilon)}\\
                                        &=C'(n,V,i_0(R-r),\varepsilon, \delta),
\end{align*}
where we use lemma \ref{lemma3} in the second inequality. Since $\{\bar{B}^g_{\frac{\delta}{3}}(x_l)\}_{\l\in I}$ are disjoint, we know
$|I|\le \frac{V}{C'(n,V,i_0(R-r),\varepsilon, \delta)}<+\infty$ and $\bar{\Sigma}^r_{\delta}$ is compact.

If we denote $\Sigma^r_{\delta}=\{x\in \Sigma^r|d_g(x,\partial \Sigma^r)>\delta\}$, then the last paragraph guarantees $\{\Sigma^{r}_{\frac{1}{j}}\}_{j=1}^{+\infty}$ is a compact exhaustion of $\Sigma^r$. Hence for $F^r:\Sigma^r\to B_r(0)$ to be proper in $B_r(0)$, it is enough to show $F^r(\partial \Sigma^{r}_{\frac{1}{j}})\subset B^{c}_{r-\frac{3}{j}}(0)$ for infinite $j$.

 Now we show $F^r(\partial \Sigma^{r}_{\frac{1}{j}})\cap B_{r-\frac{3}{j}}(0)\neq \emptyset$ will cause a contradiction. In fact, if $\exists x_0\in \partial \Sigma^r_{\frac{1}{j}}$ s.t. $y_0:=F^r(x_0)\in B_{r-\frac{3}{j}}(0)$, then $|y_0|<r-\frac{3}{j}$ and $B_{\frac{3}{j}}(y_0)\subset B_{\frac{3}{j}+|y_0|}(0)\subset B_r(0)$. Put $q_{\alpha}=\Phi_{\alpha}^{A_r}(x_0)\in U_{\alpha}^{A_r}$ and $y_{\alpha}=F_{\alpha}(q_{\alpha})$. Then $y_{\alpha}=F_{\alpha}\circ \Phi_{\alpha}^{A_r}(x_0)\to F^{r}(x_0)=y_0$ as $\alpha\to \infty$. So for $\alpha$ large enough, $y_{\alpha}\in B_{\frac{1}{j}}(y_0)$ and $B_{\frac{2}{j}}(y_{\alpha})\subset B_{\frac{3}{j}}(y_0)\subset B_r(0)$. Let $V_{\alpha}$ be the component of $F_{\alpha}^{-1}(B_{\frac{2}{j}}(y_{\alpha}))\cap \Sigma_{\alpha}$ containing $q_{\alpha}$. Then $i_{g_{\alpha}}(x)\ge i_0(R-r), \forall x\in V_{\alpha}$. Thus by the same argument as in the global case,  after passing to a subsequence again, $\tilde{X}^{A_r}_{\alpha}=(F_{\alpha}, U^{A_r}_{\alpha}\cup V_{\alpha}, g_{\alpha},p_{\alpha})$ converges to some $\tilde{X}^{r}=(\tilde{F}^r,\tilde{\Sigma}^r,g,p)\in \mathcal{C}_r$ and $X^r\le \tilde{X}^r$. Since $B_{\frac{2}{j}}^{g_{\alpha}}(q_{\alpha})\subset V_{\alpha}\subset U^{A_r}_{\alpha}\cup V_{\alpha}\to \tilde{\Sigma}^r$ and $q_{\alpha}\to x_0$, we know $B^g_{\frac{3}{2j}}(x_0)\in \tilde{\Sigma}^r$. But $x_0\in \partial\Sigma^r_{\frac{1}{j}}$ implies $d_g(x_0,\partial \Sigma^r)\le \frac{1}{j}<\frac{3}{2j}\le d_g(x_0,\partial \tilde{\Sigma}^r)$, thus $\Sigma^r\subsetneq\tilde{\Sigma}^r$  and $X^r<\tilde{X}^r$. This contradicts to $\Sigma^r$ is maximal.

 At last, we know $F^r:\Sigma^r\to B_r(0)$ is proper for each $r\in (0,R)$. Take a direct limit by diagonal argument and we get a quadruple $X^R=(F^R,\Sigma^R,g,p)\in \mathcal{C}_R$ such that $F^R:\Sigma^R\to B_R(0)$ is proper.
    \end{proof}
\section{Application \Rmnum{1}-Leon Simon's decomposition theorem}\label{section:estimate the isothermal radius}
  \subsection{Estimation of the isothermal radius}
 We now give the estimation of the isothermal radius. The key is that the $L_{loc}^p$-convergence of the metric could be improved to $C^0_{loc}$-convergence in the blowup case (Corollary \ref{uniformly convergence}) and the isothermal radius is continuous in $C^0$ topology.
 \begin{proof}[Proof of Property \ref{isothermal radius}]
 We argue by contradiction. If not, there exists a $V>0$, a sequence of immersing Riemannian surfaces $F_k:(\Sigma_k, g_k, 0)\to \mathbb{R}^n$ and $x_k\in \Sigma_k\cap B_{\frac{1}{2}}(0)$ satisfying $\mu_{g_k}(\Sigma_k\cap B_1(0))\le V$ and $\int_{\Sigma_k \cap B_1(0)}{|A_k|}^2d\mu_{g_k}=\varepsilon_k\to 0$, but $\alpha_k=i_{g_k}(x_k)=\mathop{\mathrm{inf}}_{x\in \Sigma_k\cap B_{\frac{1}{2}}}i_{g_k}(x)\to 0$. Let $\tilde{F}_k:(\tilde{\Sigma}_k,h_k,0)=(\alpha_k^{-1}(\Sigma_k-x_k), \alpha_k^{-2}g_k,0)\to\mathbb{R}^n$. Then for any fixed $R>0$, we have
 \begin{enumerate}[(1)]
 \item for any fixed $\delta>0$,
  \begin{align*}
       \mu_{h_k}(\tilde{\Sigma}_k\cap B_R(0))&=(\alpha_k R)^{-2}\mu_{g_k}(\Sigma_k\cap B_{\alpha_k R}(x_k))R^2\\
          &\le((1+\delta)\mu_{g_k}(\Sigma_k\cap B_1(0))+(1+\delta^{-1})Will(\Sigma_k))R^2\\
          &\le((1+\delta)V+2(1+\delta^{-1})\int_{\Sigma_k\cap B_1(0)}{|A_k|}^2d\mu_{g_k})R^2\\
          &\le((1+\delta)V+2(1+\delta^{-1})\varepsilon_k)R^2\\
          &\le(1+2\delta)VR^2
       \end{align*}
 for $k$ large enough, where we use the monotonicity formulae Lemma~\ref{monotonicity formulae} in the first inequality;
 \item $\int_{\tilde{\Sigma}_k\cap B_R(0)}{|\tilde{A}_k|}^2d\mu_{h_k}=\int_{\Sigma_k\cap B_{\alpha_k R}(x_k)}{|A_k|}^2d\mu_{g_k}\le \varepsilon_k\le 4\pi\varepsilon$ for $k$ large enough;
 \item $i_{h_k}(\tilde{x})=\alpha_k^{-1}i_{g_k}(x)\ge 1$ for $\tilde{x}=\alpha_k^{-1}(x-x_k)\in \tilde{\Sigma}_k\cap B_{\frac{1}{2\alpha_k}}(0)$ but $i_{h_k}(0)=\alpha_k^{-1}i_{g_k}(x_k)=1$.
 \end{enumerate}
 That is, $(\tilde{F}_k:(\tilde{\Sigma}_k, h_k,0)\to \mathbb{R}^n)\in \mathcal{C}(n,\varepsilon,1,(1+2\delta)V, R)$ for $k$ large enough. By Theorem ~\ref{convergence}, we know there exists a pointed Riemannian surface $(\Sigma_{\infty},h_{\infty},0)$ and a proper immersion $F:(\Sigma_{\infty},h_{\infty},0)\to\mathbb{R}^n$ with continuous metric $g$ such that $(\tilde{\Sigma}_k,h_k,0)\to(\Sigma_{\infty},h_{\infty},0)$ in pointed $L^p$-topology and $\tilde{F}_k\to F$ in weak $W^{2,2}$ topology. Again, the lower semi-continuity property of total curvature implies
  $$\int_{\Sigma_{\infty}\cap B_R(0)}{|A_{\infty}|}^2d\mu_{h_{\infty}}\le\liminf_{k\to \infty}\int_{\tilde{\Sigma}_k\cap B_R(0)}{|\tilde{A}_k|}^2d\mu_{h_k}=0.$$
 Thus $\Sigma_{\infty}$ is a flat plane $\mathbb{R}^2$ with isothermal radius $i_{h_{\infty}}(\Sigma_{\infty})=\infty$ and we may assume $F(x,y)=(x,y,0^{n-2})$ for some constant $a>0$. In fact, $F(\Sigma_{\infty})$ is included in a plane $P\subset\mathbb{R}^n$ since $A_{\infty}=0$. Thus $F:(\Sigma_{\infty},0)\to (P,0)$ is a proper immersion of a surface to a plane with the same dimension. Since a proper local diffeomorphsim is a covering map and $P$ is simply connected, we know $F:(\Sigma_{\infty},0)\to (P,0)$ is a diffeomorphsim.  So we can assume $\Sigma_{\infty}=\mathbb{R}^2$ and $F:\mathbb{R}^2\to P\subset \mathbb{R}^n$ has the form $F(x,y)=(x,y, 0^{n-2})$.

 Next, we observe the continuity of isothermal radius in the special case to get a contradiction. By Theorem ~\ref{intrinsic convergence}$(d)$,  $(\tilde{\Sigma}_k,h_k,0)\to(\Sigma_{\infty},h_{\infty},0)$ in $C^0_{loc}$ topology.  So,there exist local diffeomorphisms $f_k: D_{1+2\sigma}\subset \mathbb{R}^2\to \Omega_k\supset B_{1+\sigma}^{\tilde{\Sigma}_k}(0)\subset \tilde{\Sigma}_k$ such that $f_k^*h_k\xrightarrow[]{C_c^0}h_{\infty}=(dx^2+dy^2)$ on $D_{1+\sigma}$. Regard $f_k$ as the coordinate of $\Omega_k$ and write $f_k^*g_k=E_k{dx}^2+2F_kdxdy+G_k{dy}^2$. Then we have $E_k, G_k\rightrightarrows 1$ and $F_k\rightrightarrows 0$ uniformly on $D_{1+\sigma}$. Hence the Beltrami coefficient
 $$\mu_{h_k}=\frac{\sqrt{(E_k-G_k)^2+4F_k^2}}{E_k+G_k+2\sqrt{E_kG_k-F_k^2}}\rightrightarrows 0$$
 on $D_{1+\sigma}$. So, by the existence of isothermal coordinate \cite{Chern 1}(see also \cite{LZ}), we know $i_{h_k}(0)\ge 1+\sigma$. A contradiction!
 \end{proof}
 \begin{rmk}
 If we assume $U(p)\subset 101 B_r^{\Sigma}(p)$(which is satisfied in the above case) in the definition of isothermal radius, then  by similar argument as in the last paragraph, one can prove the isothermal radius is continuous in pointed $C^0$ topology. This is the advantage of the isothermal radius to the injective radius.
 \end{rmk}

  \subsection{}
  We now give a blowup approach to Leon Simon's decomposition theorem \cite{LS}. The idea is first arguing by blowup to show the immersions are indeed embeddings, then using the Poincar\'e inequality to estimate the area out of the multi-graph, and finally using the area estimate to prove the multi-graph is in fact single.
  During this subsection, we use $\Psi(\varepsilon)$ to denote a function satisfying  $\lim_{\varepsilon\to 0}\Psi(\varepsilon)=0$ which may change from line to line. We also use the notation $U \sim B_r^{\Sigma}(p)$ to mean $B^{\Sigma}_{(1-\Psi(\varepsilon))r}(p)\subset U\subset B^{\Sigma}_{(1+\Psi(\varepsilon))r}(p)$.
  \begin{thm}\label{decomposition}
  For fixed $R>0$ and any $V>0$, there exists an $\varepsilon_2=\varepsilon_2(n,V)>0$ such that for $\forall \varepsilon\in (0,\varepsilon_2]$ the following holds:

  If  $F:(\Sigma,g,p)\to (\mathbb{R}^n,0)$ is a Riemannian immersion proper in $B_R(0)$ with
  $$\mathcal{H}^2(\Sigma_k\cap B_R(0))\le VR^2,\ \int_{\Sigma\cap B_R(0)}{|A|}^2d\mu_g\le \varepsilon^2\  and \ F(\partial\Sigma)\subset B_R^c(0),$$
then there exists an topologically disk $U_0(p)\sim B_{\frac{R}{2}}^{\Sigma}(p)$ with $p\in \Sigma^{\frac{R}{2}}(0)\subset U_0(p)$ such that $F: U_0(p)\to \mathbb{R}^n$ is an embedding with $F(\partial U_0)\subset B^c_{\frac{R}{2}}(0)$.

   Moreover, there exist finite many disjoint topologically disks ${D_i}$ contained in $U_0$, such that
   $$\Sigma_i\mathcal{H}^2(D_i)\le \Psi(\varepsilon),\ \Sigma_i \mathcal{H}^1(\partial D_i)\le \Psi(\varepsilon),$$
   and
    $$U_0(p)\backslash \cup_iD_i=\mathrm{graph}u=\{(x,u(x))|x\in \Omega\}$$
    for some function $u:\Omega\to L^{\bot}$ with $|\nabla u|\le \Psi(\varepsilon)$, where $L\subset \mathbb{R}^n$ is an affine 2-plane, and $\Omega\subset L$ is a bounded domain with the topology type of a disk with finite many disjoint small disks removed.
  \end{thm}
  \begin{proof}
  Without loss of generality, we assume $R=1$ and argue by contradiction. So assume there are a sequence of Riemannian immersions $F_k:(\Sigma_k,g_k,p_k)\to(\mathbb{R}^n,0)$ proper in $B_1(0)$ and satisfying $\int_{\Sigma_k\cap B_1(0)}{|A_k|}^2d\mu_{g_k}=\varepsilon^2_k\to0$, $\mathcal{H}^2(\Sigma_k\cap B_1(0))\le V$ but the conclusion of the theorem does not hold. We will prove the conclusion does hold for $k$ large enough and get a contradiction.

   First of all, Property~\ref{isothermal radius} implies there exists $\alpha_0(V)>0$ such that $i_{g_k}(x)\ge \alpha_0(V)$, $\forall x\in \Sigma_k\cap B_{\frac{3}{4}}(0)$. Then by Theorem ~\ref{convergence}, there exists an immersing Riemannian surface $F_{\infty}:(\Sigma_{\infty},g_{\infty},p)\to\mathbb{R}^n$ which is properly immersed in $B_{\frac{2}{3}}(0)$ such that $(\Sigma_k,g_k,p_k)\to (\Sigma_{\infty},g_{\infty},p)$ in pointed $C^0$ topology, $F_k\to F_{\infty}$ in pointed $W^{1,p}$ and weak $W^{2,2}$ topology and
    $$\int_{\Sigma_{\infty}\cap B_{\frac{2}{3}}(0)}{|A|}^2d\mu_{g_{\infty}}\le\liminf_{k\to \infty}\int_{\Sigma_k\cap B_{\frac{2}{3}}(0)}{|A_k|}^2d\mu_{g_k}=0.$$
     Thus $F_{\infty}(\Sigma_{\infty})\cap B_{\frac{2}{3}}(0)\subset\mathbb{R}^n$ is a flat disk in some 2-dimensional linear subspace $P$. Since $F_{\infty}$ is a proper immersion in $B_{\frac{2}{3}}(0)\cap P$, by the simply connected property of $B_{\frac{2}{3}}(0)\cap P$, we may assume $F_{\infty}$ is an embedding with the form $F_{\infty}(x_1,x_2)=(x_1,x_2,0^{n-2})$  and $F_k:D_{\frac{2}{3}}(0)\to \mathbb{R}^n$ satisfies
   $$F_k^*g_{\mathbb{R}^n}\xrightarrow[]{C^0(D_{\frac{2}{3}})} dx_1^2+dx_2^2\ and\ \|F_k-F_{\infty}\|_{W^{1,p}(D_{\frac{2}{3}}(0))}+\|F_k-F_{\infty}\|_{C^{\beta}(\bar{D}_{\frac{3}{5}})(0)}\to0,$$
   for $\forall \beta\in(0,1)$. Moreover, we have $\bigcup_{r<\frac{2}{3}}F_k(D_r)\supset\Sigma_k^{\frac{2}{3}}(p_k)$ and may assume $\bigcup_{r<\frac{3}{5}}F_k(D_r)\supset \Sigma_k^{\frac{3}{5}}(p_k)$ by taking a subsequence. So for $k$ large enough, there exists $r_k<\frac{3}{5}$ s.t. $\Sigma_k^{\frac{1}{2}}(p_k)\subset F_k(D_{r_k})\subset F_k(\bar{D}_{\frac{3}{5}})$.

    Step 1. $F_k:D_{\frac{3}{5}}\to \mathbb{R}^n$ is embedding for $k$ large.

    To prove this, we also argue by contradiction(similar to \cite[Lemma 3.6]{Jingyi Chern Yuxiang Li}, see also \cite[Lemma 2.1.3]{MWH}). If it is not this case, then there exist two sequences of points $\{x_k\neq y_k\}\subset D_{\frac{3}{5}}$ such that $F_k(x_k)=F_k(y_k)$. By $\|F_k-F_{\infty}\|_{C^{\beta}(\bar{D}_{\frac{3}{5}})(0)}=:c_k\to0$, we have $$|x_k-y_k|=|(F_k(x_k)-F_{\infty}(x_k))-(F_k(y_k)-F_{\infty}(y_k))|\le c_k|x_k-y_k|^{\beta},$$
     and hence $|x_k-y_k|^{1-\beta}\le c_k\to 0$. Intuitively, this means the sequence $F_k$ will blowup to a map degenerating on the limit direction of $x_k-y_k$, which contradicts to the immersion assumption. More precisely, we may assume $x_k,y_k\to x_0\in \bar{D}_{\frac{3}{5}}$. Let $y_k=x_k+r_ke_k$, where $r_k=|x_k-y_k|$ and $|e_k|=1$. Assume $e_k\to e=(0,1)\in S^1$ and define
  $$f_k(x)=\frac{F_k(x_k+r_kx)-F_k(x_k)}{r_k}.$$
  Since $\int_{\Sigma_k} |A_k|^2\to 0$, so by the same argument as in the proof of Proposition ~\ref{isothermal radius}, we could assume $F_k^*g_{\mathbb{R}^n}=E_kdx_1^2+2T_kdx_1dx_2+G_kdx_2^2$ with $E_k,G_k\rightrightarrows 1$, and $T_k\rightrightarrows0$ uniformly in $D_{\frac{2}{3}}$. Thus $\|dF_k\|_{L^{\infty}(D_{\frac{2}{3}})}\le 3$.
  So
  \begin{align*}|f_k(e)|=\frac{|F_k(x_k+r_ke)-F_k(x_k)|}{r_k}&=\frac{|F_k(x_k+r_ke)-F_k(x_k+r_ke_k)|}{r_k}\\
                                                             &\le\|d F_k\|_{L^{\infty}(D_{\frac{2}{3}})}|e_k-e|\to 0.
  \end{align*}
  Furthermore, we have $f_k(0)=0$,
  $$\|df_k\|_{L^{\infty}(D_1)}=\|dF_k\|_{L^{\infty}(D_{r_k}(x_k))}\le\|dF_k\|_{L^{\infty}(D_{\frac{1}{15}}(x_0))}\le 3$$
  and
  $$\int_{D_1}|d^2f_k|^2=\int_{D_{r_k}(x_k)}|d^2F_k|^2(y)dy\to 0,$$
  which imply $f_k\to f(x_1,x_2)=ax_1+bx_2+c$ in $W^{2,2}(D_1)\cap W^{1,p}(D_1)\cap C_c^{\beta}(D_1)$. So, $f(0)=0$ and $f(e)=0$ imply $b=c=0$ and $f(x_1,x_2)=ax_1$. But on the other hand, for $a.e.x$,
  $$df\otimes df(x)=\lim_{k\to \infty}df_k\otimes df_k(x)=\lim_{k\to \infty}dF_k\otimes dF_k(x_k+r_kx)\ge \frac{1}{2}(dx_1^2+dx_2^2),$$
 i.e., $b\neq 0$. A contradiction!

Step 2. Area estimate out of the graph.

W.O.L.G., we can assume $F_k:D_{\frac{3}{5}}(0)\to \mathbb{R}^n$ is a conformal embedding for $k$ large, i.e., $F_k^*g_{\mathbb{R}^n}=e^{2u_k}(dx^1\otimes dx^1+dx^2\otimes dx^2)$. In fact, since
$F_k^*g_{\mathbb{R}^n}=E_kdx_1^2+2T_kdx_1dx_2+G_kdx_2^2$ with $E_k,G_k\rightrightarrows 1$, and $T_k\rightrightarrows0$ uniformly in $D_{\frac{2}{3}}$, by the existence of isothermal  coordinates, there exists a $\varphi_k:D_{\frac{2}{3}}(0)\to D_{\frac{2}{3}}(0)$ such that $\varphi_k(0)=0$ and
$$\varphi_k^*F_k^*g_{\mathbb{R}^n}=e^{2u_k}(dx^1\otimes dx^1+dx^2\otimes dx^2).$$
If we define $\tilde{F}_k=F_k\circ \varphi_k$ and $\tilde{g}_k=e^{2u_k}(dx^1\otimes dx^1+dx^2\otimes dx^2)$, then $\tilde{F}_k:D_{\frac{2}{3}}(0)\to \mathbb{R}^n$ is a conformal embedding with $vol_{\tilde{g}_k}(D_{\frac{2}{3}}(0))=vol_{g_k}(F_k(D_{\frac{2}{3}}(0))$ and $\int_{D_{\frac{2}{3}}(0)} |\tilde{A}_k|^2_{\tilde{g}_k}=\int_{F_k(D_{\frac{2}{3}}(0))}|A_k|^2_{g}$. So $F_k$ could be replaced by $\tilde{F}_k$ and all the above arguments hold.

Take $i_0=\frac{11}{20}<\frac{3}{5}$, then $W_k=F_k(D_{i_0})\sim B_{i_0}^{\Sigma_k}(p_k)$. For $\forall x\in D_{i_0}$, consider the oriented Gauss map $G_k(x)=e^{-2u_k}\frac{\partial F_k}{\partial x_1}(x)\wedge \frac{\partial F_k}{\partial x_2}(x)$, and let $f_k(x)=|G_k(x)-V_k|$,  where $V_k=\frac{1}{\pi i_0^2}\int_{D_{i_0}}G_k(x)dx$.
 Then
  $$|\nabla^{\Sigma_k} f_k|_{g_k}(x)\le |\nabla^{\Sigma_k} G_k|_{g_k}(x)=|A_k|_{g_k}(x).$$
 So, for $\forall \eta_k >0, \exists t_k\in (\eta_k/2,\eta_k)$ such that $t_k$ is regular value of $f_k$(i.e., $\Gamma_k=f_k^{-1}(t_k)$ consists of some circles and arcs in $D_{i_0}$), and
 \begin{align*}
 \mathcal{H}^1(\Gamma_k)&\le\frac{2}{\eta_k-\frac{\eta_k}{2}}\int_{\frac{\eta_k}{2}}^{\eta_k}\mathcal{H}^1\{f_k=t\}dt\\
              &\le\frac{4}{\eta_k}\int_{\frac{\eta_k}{2}}^{\eta_k}\frac{|A_k|_{g_k}}{|\nabla^{\Sigma_k} f_k|_{g_k}}(x)\mathcal{H}^1\{f_k=t\}dt\\
              &=\frac{4}{\eta_k}\int_{\frac{\eta_k}{2}\le f_k \le \eta_k}|A_k|_{g_k}(x)e^{2u_k}dx\\
              &\le\frac{4}{\eta_k}(\int_{D_{i_0}}|A_k|_{g_k}^2e^{2u_k})^{\frac{1}{2}}(\int_{D_{i_0}}1^2 e^{2u_k}dx)^{\frac{1}{2}}\\
              &\le \frac{8\sqrt{\pi}}{\eta_k}i_0 \varepsilon_k.
 \end{align*}
 Let $S_k=\{x\in D_{i_0} | f_k(x)>t_k\}$. Then by the Poincar\'{e} inequality \cite[section 7.8, page164]{GT}, we get
 \begin{align*}
 (\int_{D_{i_0}}|f_k-\frac{1}{|S_k|}\int_{S_k}f_k|^2)^{\frac{1}{2}}\le \frac{4\sqrt{2\pi} i_0^2}{|S_k|^{\frac{1}{2}}}(\int_{D_{i_0}}|\nabla f_k|^2)^{\frac{1}{2}}\le \frac{4\sqrt{2\pi} i_0^2}{|S_k|^{\frac{1}{2}}} \varepsilon_k,
 \end{align*}
 where we use the conformal invariance of the Dirichlet integral in the last inequality. But we also have,
 \begin{align*}
 (\int_{S_k}|f_k-\frac{1}{|S_k|}\int_{S_k}f_k|^2)^{\frac{1}{2}}\ge \|\frac{1}{|S_k|}\int_{S_k}f_k\|_{L^2(S_k)}-\|f_k\|_{L^2}\ge t_k|S_k|^{\frac{1}{2}}-\|f_k\|_{L^2},
 \end{align*}
 and
 \begin{align*}
 \|f_k\|_{L^2(S_k)}&\le (\int_{D_{i_0}}|G_k(x)-\frac{1}{\pi i_0^2}\int_{D_{i_0}}G_k|^2)^{\frac{1}{2}}\\
                   &\le 4\sqrt{n(n-1)}i_0(\int_{D_{i_0}}|\nabla G_k(x)|^2)^{\frac{1}{2}}\le 4\sqrt{n(n-1)}i_0 \varepsilon_k,
 \end{align*}
 where we use the conformal invariance of the Dirichlet integral again. Thus we have
$$t_k|S_k|^{\frac{1}{2}}\le \frac{4\sqrt{2\pi}i_0^2}{|S_k|^{\frac{1}{2}}} \varepsilon_k+\frac{4\sqrt{n(n-1)}i_0}{\sqrt{\pi}i_0} \varepsilon_k,$$
thus if we take $\eta_k=\sqrt{\varepsilon_k}$, then $t_k\in(\frac{\sqrt{\varepsilon_k}}{2}, \sqrt{\varepsilon_k})$ and
$$|S_k|\le (64+8\sqrt{n(n-1)\pi})\varepsilon_k^{\frac{1}{2}}i_0^2.$$
So, when taking $\eta_k=\varepsilon_k^{\frac{1}{2}}$, we have  the length estimate and the area estimate:
\begin{align*}
\mathcal{H}^1(\Gamma_k)\le 8\sqrt{\pi}\varepsilon_k^{\frac{1}{2}}i_0\ \ \  , \ \ \ |S_k|\le (64+8\sqrt{n(n-1)\pi})\varepsilon_k^{\frac{1}{2}}i_0^2
\end{align*}

  Furthermore, since $|G_k(x)-V_k|=f_k(x)\le t_k\le\varepsilon_k^{\frac{1}{2}}\ll \sqrt{2}$ for $\forall x\in D_{i_0}\backslash S_k$, $W_k\backslash F_k(S_k)$ could be written as a multi-graph with gradient small. More precisely, if we define $L_k=[V_k]$ to be the linear space corresponding to $V_k\in\Lambda^2(\mathbb{R}^n)$ and $\pi_k: \mathbb{R}^n\to L_k$ the orthogonal projection, $\Omega^{\prime}_k=\pi_k(W_k\backslash F_k(S_k))$, then $\pi_k$ is local diffeomorphism when restricted on $W_k\backslash F_k(S_k)$:

  If not, there must exists $x\in W_k\backslash F_k(S_k)$ and $e_1(x)\in [G_k(x)]\cap ker d\pi=[G_k(x)] \cap L_k^{\bot}$, that is, for any orthogonal basis $e_1^0, e_2^0$ of $L_k$, $\langle e_1,e_1^0\rangle=\langle e_1, e_2^0\rangle=0$. Choose any $e_2\in [G_k(x)]$ such that $[G_k(x)]=[e_1\wedge e_2]$, then $\langle e_1\wedge e_2, e_1^0\wedge e_2^0\rangle=0$, and $|G(x)-V_k|=\sqrt{|e_1\wedge e_2|^2+|e_1^0\wedge e_2^0|^2}=\sqrt{2}$.

   Hence there is a multi-function $v_k^{\prime}:\Omega_k^{\prime}\to L_k^{\bot}$ with gradient $|\nabla v_k^{\prime}|\le\Psi(t_k)=\Psi(\varepsilon_k)$ such that $W_k=$ graph$v_k^{\prime}$. But when noticing that
   $$V_k=\frac{1}{\pi i_0^2}\int_{D_{i_0}}G_k\to\frac{1}{\pi i_0^2}\int_{D_{i_0}}G_{\infty}=\frac{\partial}{\partial x_1}\wedge \frac{\partial}{\partial x_2}=V_{\infty}$$
   (by the control convergence theorem), we know that Graph$v_k^{\prime}$ could be written as a Graph over $L_{\infty}=[V_{\infty}]$ with gradient $\le |\nabla v_k^{\prime}|+|V_k-V_{\infty}|=\Psi(\varepsilon_k)$. That is:
 \begin{align*}
 W_k\backslash F_k(S_k)=\mathrm{Graph} \ \mathrm{of}\   v_k: \Omega^2_k\subset V_{\infty}\to V_{\infty}^{\perp},\ \ |\nabla v_k|\le\Psi(\varepsilon_k),
 \end{align*}
 where $\Omega^2_k$ is a bounded planner domain contained in $\Omega^1_k=\pi(W_k)\sim D_{i_0}$

 The domain $W_k\backslash F_k(S_k)$ may not be connected since there are some arcs in $\Gamma_k$ with ends on $\partial W_k$. But when noticing $\mathcal{H}^1(\Gamma_k)\le 8\sqrt{\pi}\varepsilon_k^{\frac{1}{2}}i_0$, $|S_k|\le (64+8\sqrt{n(n-1)\pi})\varepsilon_k^{\frac{1}{2}}i_0^2$ and $F_k\to F_{\infty}$,  we know $W_k\backslash F_k(S_k)$ contains a unique large connected  component $C_k^{\prime} \sim W_k\backslash F_k(S_k)$ such that
 \begin{align*}
 (1-\Psi(\varepsilon_k))\pi i_0^2\le \mathcal{H}^2(C_k^{\prime})\le(1+\Psi(\varepsilon_k))\pi i_0^2.
 \end{align*}
 W.L.O.G., we can assume $\Gamma_k$ does not contain any arc and $C_k^{\prime}=W_k\backslash F_k(S_k)$.

Step 3. The multi-graph is in fact single valued.

Consider $\Omega^2_k:=\pi(C_k^{\prime})$, where the orthogonal projection $\pi:\mathbb{R}^n\to L_{\infty}$ is local diffeomorphism when restricted on $C_k^{\prime}$. Since $F_k\to F_{\infty}$ in $C^{\alpha}(D_{i_0};\mathbb{\mathrm{R}}^n)$, we know $\pi\circ F_k\to Id|_{D_{i_0}}$ in $C^{\alpha}(D_{i_0};\mathbb{\mathrm{R}}^2)$ and hence
  $$\Omega_k^2=\pi(C_k^{\prime})=\pi\circ F_k(D_{i_0}\backslash S_k)\supset D_{(1-\Psi(\varepsilon_k))i_0}\backslash\pi\circ F_k(S_k)=:\Omega_k^3.$$
Since $\pi$ is local diffeomorphism when restricted to $C_k^{\prime}$, we know $\partial \pi\circ F_k(S_k)\subset \pi(\partial F_k(S_k))=\pi(\Gamma_k)$ and
  $$\mathcal{H}^1(\partial \pi\circ F_k(S_k))\le \mathcal{H}^1(\pi(\Gamma_k))\le\mathcal{H}^1(\Gamma_k)\le8\sqrt{\pi}\varepsilon_k^{\frac{1}{2}}i_0,$$
 thus we may assume
  $$\Omega_k^3\supset D_{(1-\Psi(\varepsilon_k))i_0}\backslash \cup_{i\in I_k}d_i=:\Omega_k,$$
where $\{d_i\}_{i\in I_k}$ is a finite set of disjoint closed topological disks in $D_{(1-\Psi(\varepsilon_k))i_0}$ such that $\Sigma_{i\in I_k} \mathcal{H}^1(\partial d_i)\le 8\sqrt{\pi}\varepsilon_k^{\frac{1}{2}}i_0$ and so $\Sigma_{i\in I_k}|d_i|\le 16 \varepsilon_k i_0^2 $ by the isoperimetric inequality.

 Now, let $C_k=\pi^{-1}(\Omega_k)\cap C_k^{\prime}$. Then also because $\pi$ is local diffeomorphism on $C_k^{\prime}$, there are finite many disjoint closed topological disks $\{D_i\}_{i\in J_k}\subset W_k$  with $\Sigma_{i\in J_i}\mathcal{H}^1(\partial D_i)\le (1+\Psi(\varepsilon_k))(8\sqrt{\pi}\varepsilon_k^{\frac{1}{2}}i_0)$ such that $$C_k=C_k^{\prime}\backslash \cup_{i\in J_k} D_i=W_k\backslash \cup_{i\in J_k} D_i$$
 and
 $$\pi(\partial C_k)=\pi(\cup_{i\in J_k} \partial D_i)=\partial \Omega_k,$$
  that is, $\pi: C_k\to \Omega_k$ is a proper local diffeomorphism and hence a covering map.
But on the other hand, $C_k\subset C_k^{\prime}$ implies
  $$\mathcal{H}^2(C_k)\le (1+\Psi(\varepsilon_k))\pi i_0^2$$
 but
 $$|\Omega_k|=|D_{(1-\Psi(\varepsilon_k))i_0}\backslash \cup_{i\in I_k}d_i|\ge (1-\Psi(\varepsilon_k))\pi i_0^2-16 \varepsilon_k i_0^2=(1-\Psi(\varepsilon_k))\pi i_0^2.$$
  Thus $\pi: C_k\to \Omega_k$ must be a single cover, i.e., $v_k$ is single valued when restricted on $\Omega_k$.  Moreover, the uniform convergence of the metrics induced by  $F_k:D_{i_0}\to W_k$ again implies $\Sigma_{i\in J_k}\mathcal{H}^2(D_i)\le (1+\Psi(\varepsilon_k))^3 8\sqrt{\pi}\varepsilon_k^{\frac{1}{2}}i_0^2$.

Letting $U_0(p_k)=W_k$, $u_k=v_k|_{\Omega_k}$ we get the conclusion!
\end{proof}

From the Decomposition theorem, we can get the following ``area drop" property.
\begin{cor}[$\mathbf{Density\  Drop}$]
For fixed any $V>0$, there exists an $\varepsilon_2=\varepsilon_2(n,V)>0$ such that for $\forall \varepsilon\in (0,\varepsilon_2]$ the following holds:

  If  $F:(\Sigma,g,p)\to (\mathbb{R}^n,0)$ is a Riemannian immersion proper in $B_R(0)$ with
  $$\mathcal{H}^2(\Sigma_k\cap B_R(0))\le VR^2\ \ and \ \  \int_{\Sigma\cap B_R(0)}{|A|}^2d\mu_g\le \varepsilon^2,$$
 then
  $$\frac{\mathcal{H}^2(\Sigma^{\frac{R}{2}}(0))}{\pi (\frac{R}{2})^2}\le 1+\Psi(\varepsilon).$$
\end{cor}
\begin{proof}
Since the Gradient of the graph map and the measure outside the graph are both well estimated, we get the result.
\end{proof}
\begin{rmk}
With this dropped area density, by a similar argument as in the first step of the proof of the Allard Regularity theorem--the Lipschitz Approximation--and Reifenberg's Topological disk theorem, we know there exists some $\Psi_1(\varepsilon)=\frac{1}{2^{19}}\min\{\Psi(\varepsilon)^{\frac{1}{2}},\varepsilon\}$, such that for any $\xi\in \Sigma^{\Psi_1(\varepsilon)R}(p)$, $\Sigma^{\Psi_1(\varepsilon)R}(\xi)$ is a topological disk. This means, in the small scale of $\Psi_1(\varepsilon)R$, there is no holes caused by intersecting $\Sigma$ with an extrinsic ball $B_{\Psi_1(\varepsilon)R}(0)$.
\end{rmk}
\section{Application \Rmnum{2}-H\'{e}lein's convergence theorem}\label{section:non-collapsing Helein}
\subsection{}
H\'{e}lein's convergence theorem was first proved by H\'{e}lein \cite{H}. An optimal version of the theorem was stated in \cite{KL12}  as following:
\begin{thm}\label{Helein's ocnvergence theorem}
 Let $f_k\in W^{2,2}(D,\mathbb{R}^n)$ be a sequence of conformal immersions with induced metric $g_{k}=e^{2u_k}\delta_{ij}$ and satisfy
$$
\int_{D}|A_{f_k}|^2d\mu_{g_k}\le\gamma<\gamma_n=
\begin{cases}
8\pi \text{ for } n=3, \\
4\pi \text{ for } n\ge4.
\end{cases}
$$
Assume also that $\mu_{g_k}(D)\le C$ and $f_k(0)=0$. Then $f_k$ is bounded in $W^{2,2}(D_r,\mathbb{R}^n)$ for any $r\in (0,1)$, and there is a subsequence  such that one of the following  two alternatives holds:
\begin{enumerate}[$(a)$]
\item $u_k$ is bounded in $L^{\infty}(D_r)$ for any $r\in (0,1)$, and $f_k$ converges weakly in $W^{2,2}_{loc}(D,\mathbb{R}^n)$ to a conformal immersion $f\in W^{2,2}_{loc}(D,\mathbb{R}^n)$;
\item $u_k\to -\infty$ and $f_k\to 0$ locally uniformly on $D$.
\end{enumerate}
\end{thm}
\begin{cor}\label{non-collapsing}
Let $f_k\in W^{2,2}(S^2,\mathbb{R}^n)$ be a sequence of immersions with $g_k=df_k\otimes df_k$ and satisfy
$Will(f_k)\le C$. If $f_k(x_0)=0$ for some $x_0\in S^2$  and $$vol_{g_k}(S^2)\equiv V.$$ Then there exist a subsequence$($still denoted as $f_k)$ and a sequence of M\"{o}bius transformations $\phi_k\in \mathcal{M}(S^2)$ such that $\tilde{f}_k=f_k\circ \phi_k$ converges weakly to some immersion $f\in W^{2,2}_{loc}(S^2\backslash S,\mathbb{R}^n)$ in $W^{2,2}_{loc}(S^2\backslash S, \mathbb{R}^n)$, where
$$S=\{x\in S^2|\lim_{r\to 0}\liminf_{k\to \infty}\int_{B_r^{g_0}(x)}|A_{\tilde{f}_k}|^2d\mu_{\tilde{g}_k}\ge \gamma_n\}$$
is the finite set of singular points in $S^2$.
\end{cor}
\begin{proof}
Let $\mu_k=\mathcal{H}^2\llcorner f_k(S^2)$ and $\nu_k=\mu_k\llcorner |A_k|^2$. Then $\nu_k\ll\mu_k$, $\mu_k(\mathbb{R}^n)=V$ and $\nu_k(\mathbb{R}^n)\le 4\pi+C$ by the Gauss-Bonnet formulae. So after passing to subsequences, $\mu_k\rightharpoonup\mu$ and $\nu_k\rightharpoonup\nu$ as Radon measures $(\nu\ll\mu)$.  Let $d_k=diam(f_k(S^2))=\sup\{|p-q||p,q\in f_k(S^2)\}$. We claim:
\begin{enumerate}[(1)]
\item
$d_k\ge d:=\min\{1,\big(2(1+\frac{\pi C}{V})\big)^{-\frac{1}{2}}\}$
\item $\text{ spt} \mu$ contains at least infinite points,
\item for $\forall p_k\in f_k(S^2)$ and $r<\frac{d}{2}$,
$$\frac{\mu_k(B_r(p_k))}{\pi r^2}\le \theta:=\frac{6V}{\pi d^2}+2C.$$
\item $\exists p\in spt \mu$ and $r\in (0,\frac{d}{2})$, s.t. $\nu(B_r(p))\le \frac{1}{2}\varepsilon_0(\theta)$, where $\varepsilon_0(\theta)$ is the small constant defined in Proposition \ref{isothermal radius}.
\end{enumerate}
In fact, since $0\in f_k(S^2)=spt \mu_k$, if $d_k\le 1$, then by the monotonicity formulae, $\forall \delta>0$,
\begin{align*}
\frac{V}{\pi d_k^2}=\frac{\mathcal{H}^2(f_k(S^2))\cap B_{d_k}(0)}{\pi d_k^2}
&\le (1+\delta)\frac{\mathcal{H}^2(f_k(S^2))\cap B_1(0)}{\pi}+(1+\frac{1}{4\delta})Will(f_k)\\
&\le \frac{(1+\delta)V}{\pi}+(1+\frac{1}{4\delta})C.
\end{align*}
Taking $\delta=\frac{1}{2}$, we get $d_k\ge d(C,V)$.
Moreover, if $spt \mu=\{p_1,p_2,\ldots p_l\}$ is a finite set, then for each $p_i\in spt \mu$, by the monotonicity formulae again, we know for $r\le 1$,
\begin{align*}
V=\lim_{k\to \infty}\mu_k(f_k(S^2))&\le \lim_{k\to \infty} \mu_k(\mathbb{R}^n\backslash \cup_{i=1}^{l} B_{r}(p_i))+\lim_{k\to \infty}\Sigma_{i=1}^{l}\mu_k(B_r(p_i))\\
&= \mu(\mathbb{R}^n\backslash \cup_{i=1}^{l} B_{r}(p_i))+\lim_{k\to \infty}\Sigma_{i=1}^l\frac{\mu_k(B_r(p_i))}{\pi r^2}\pi r^2\\
&\le \pi r^2\lim_{k\to \infty}\Sigma_{i=1}^l\{(1+\delta)\frac{\mu_k(B_{1}(p_i))}{\pi}+(1+\frac{1}{4\delta})Will(f_k)\}\\
&\le\big((1+\delta)\frac{V}{\pi}+(1+\frac{1}{4\delta})C\big ) \pi l r^2,
\end{align*}
letting $r\to 0$ we get $V=0$. A contradiction! $(3)$ is also a simple corollary of the monotonicity formulae for $\delta=\frac{1}{2}$.

Now, let $N:=[\frac{4\pi+C}{\epsilon_0/2}]+1$ and take $\{p_i\}_{i=1}^N\subset spt\mu$. Then $\exists r>0$(we can assume $r<\frac{d}{2}$) s.t. $B_r(p_i)\cap B_r(p_j)=\emptyset$, $\forall 1\le i\neq j\le N$. Hence $\Sigma_{i=1}^N\nu(B_r(p_i))\le\nu(\mathbb{R}^n)\le 4\pi+C$ and there exists $1\le i_0\le N$, s.t. $\nu(B_r(p_{i_0}))\le \frac{4\pi+C}{N}\le\frac{\varepsilon_0}{2}$.\\
Moreover, since $p_{i_0}\in spt \mu$, there exists $p_k\in f_k(S^2)$ s.t. $|p_k-p_{i_0}|\to 0$, thus $\nu_k(B_{\frac{r}{2}}(p_k))\le \nu(B_r(p))\le \frac{\varepsilon_0}{2}$ for $k$ large enough. Similar to $(1)$, if we define $d_k(p_k)=\sup\{|p-p_k||p\in f_k(S^2)\}$, then $d_k(p_k)\ge d$ hence $f_k:S^2\to B_{\frac{r}{2}}(p)\subset B_{\frac{d}{2}}(p)$ properly for $k$ large enough. So, by Proposition \ref{isothermal radius}, for $\forall x\in f_k^{-1}(B_{\frac{r}{4}}(p_k))$, the isothermal radius
 $$i_{g_k}(x)\ge \alpha_0(\theta)\frac{r}{2}=:r_1>0.$$
Especially, we could choose $\{x_k^i\}_{i=0}^{2}\subset f_k^{-1}(B_{\frac{r}{4}}(p_k))$ such that $f_k(x^0_k)=p_k$ and $|f_k(x^i_k)-f_k(x^0_k)|=\frac{r_1}{2}$ for $i=1,2$.

Fix three different points $\{x_0^i\}_{i=0}^2\subset S^2$, there exists a unique M\"obius transformation $\phi_k\in \mathcal{M}(S^2)$ s.t. $\phi_k(x^i_0)=x_k^i, i=0,1,2.$ Consider $\tilde{f}_k=f_k\circ \phi_k:S^2\to \mathbb{R}^n$ and $\tilde{g}_k=\phi_k^*g_k$. Let
$$S=\{x\in S^2|\lim_{r\to 0}\liminf_{k\to \infty}\int_{B_r^{g_0}(x)}|A_{\tilde{f}_k}|^2dvol_{\tilde{g}_k}\ge \gamma_n\}.$$
Then $S$ is a finite set and by Theorem \ref{Helein's ocnvergence theorem}, $\tilde{f}_k$ converges to some $f$ weakly in $W^{2,2}_{loc}(S^2\backslash S, \mathbb{R}^n)$ such that either $f$ is an immersion or $f(S^2\backslash S)\equiv p$. But by the choice of $x_0^i$, we know $x_0^i\notin S, i=0,1,2$ and $|\tilde{f}_k(x_0^0)-\tilde{f}_k(x_0^i)|=\frac{r_1}{2}>0, i=1,2$, thus $f(S^2\backslash S)\neq p$ and $f:S^2\backslash S \to \mathbb{R}^n$ is an immersion.
\end{proof}

\end{document}